\documentclass[letterpaper,oneside]{amsart}
\usepackage{xcolor}

\usepackage{geometry}

\usepackage[normalem]{ulem}
\usepackage{amsthm}

\usepackage{amsmath, amssymb,mathtools}
\usepackage{mathrsfs}
\usepackage{hyperref}
\usepackage{graphicx, psfrag, epstopdf}
\usepackage{url}
\usepackage{enumitem}
\usepackage{comment}

\usepackage[backend=biber,sorting=nty,sortcites=true,style=alphabetic,
doi=true,natbib=true,isbn=false,url=false,maxbibnames=99]{biblatex}
\renewbibmacro{in:}{}

\DeclareBibliographyCategory{badbreaks}

\AtEveryBibitem{
  \ifcategory{badbreaks}
    {\defcounter{biburllcpenalty}{9000}}
    {}}

\theoremstyle{definition}
\newtheorem{definition}{Definition}[section]
\newtheorem{theorem}[definition]{Theorem}
\newtheorem{lemma}[definition]{Lemma}
\newtheorem{corollary}[definition]{Corollary}
\newtheorem{proposition}[definition]{Proposition}

\newtheorem{remark}[definition]{Remark}
{Example}

\numberwithin{equation}{section}

\def\DS{\displaystyle}

\newcommand{\Diff}{\operatorname{Diff}}

\newcommand{\im}{\operatorname{Im}}

\newcommand{\vol}{\operatorname{vol}}

\newcommand{\mc}[1]{\mathcal{#1}}

\newcommand{\cU}{\mc{U}}
\newcommand{\mf}[1]{\mathfrak{#1}}

\newcommand{\abs}[1]{\left| #1\right|}

\newcommand{\PP}{\mathbb{P}}
\newcommand{\EV}{\mathbb{E}}
\newcommand{\R}{\mathbb{R}}
\newcommand{\N}{\mathbb{N}}

\newcommand{\bd}{{\boldsymbol{d}}}
\newcommand{\Gr}{\operatorname{Gr}}
\newcommand{\SO}{\operatorname{SO}}
\newcommand{\supp}{\operatorname{supp}}

\newcommand{\fk}{{\mathfrak{k}}}
\newcommand{\fe}{{\mathfrak{e}}}

\newcommand{\graph}{{\mathrm{graph}}}
\newcommand{\cG}{{\mathcal{G}}}

\newcommand\con{{\mathfrak{m}}}

\newcommand{\cB}{{\mathcal{B}}}
\newcommand{\cF}{{\mathcal{F}}}
\newcommand{\cL}{{\mathcal{L}}}
\newcommand{\eps}{{\varepsilon}}

\newcommand{\fC}{{\mathfrak{C}}}

\newcommand{\fD}{{\mathfrak{D}}}

\newcommand{\cR}{{\mathcal{R}}}

\newcommand{\reg}{{\mathbf{k}}}

\newcommand{\dist}{{\mathbf{d}}}

\newcommand{\wt}[1]{\widetilde{#1}}

\newcommand{\bp}{{\boldsymbol{p}}}

\newcommand{\bA}{{\boldsymbol{A}}}

\newcommand{\ary}{\begin{eqnarray}}
	\newcommand{\eary}{\end{eqnarray}}

\newcommand{\aryst}{\begin{eqnarray*}}
	\newcommand{\earyst}{\end{eqnarray*}}

\newcommand{\enmt}{\begin{enumerate}}
	\newcommand{\eenmt}{\end{enumerate}}

\newcommand{\cV}{{\mathcal{V}}}

\theoremstyle{definition}

\title[Ergodicity of (co)expanding on average dynamical systems]{Ergodicity of (co)expanding on average random dynamical systems}
\author{Jonathan DeWitt}
\address{Department of Mathematics, The Pennsylvania State University, State College, PA 16802, USA}
\email{dewitt@psu.edu}
\author{Dmitry Dolgopyat}
\address{Department of Mathematics, The University of Maryland, College Park, MD 20742, USA}
\email{dolgop@umd.edu}
\author{Zhiyuan Zhang}
\address{
	Imperial College London, 
	London SW7 2AZ, 
	United Kingdom }
\email{zhiyuan.zhang@imperial.ac.uk}
\date{\today}

\begin{document}
\begin{abstract}
We prove ergodicity for random dynamics satisfying some expansion and irreducibility conditions. As an application we show that if $R_1,R_2\in \SO(d+1)$, $d\ge 2$, generate a dense subgroup, then the random dynamics of $R_1$ and $R_2$ on $S^d$ is stably ergodic. Previously this was only known  in even dimensions.  As a consequence, we deduce spectral gap and statistical limit theorems 
for such systems. In particular, our results apply 
in the presence of zero Lyapunov exponents.
\end{abstract}

\maketitle

\section{Introduction}

In this paper, we obtain a new dynamical criterion for the ergodicity of conservative random dynamical systems. Essentially, the criterion says that if the strong stable manifolds of the system are sufficiently random and the system on average expands volume on codimension $1$ planes, then volume is ergodic. The argument proceeds by studying the density points of an ergodic component, and the novelty in the argument is the application of a curved multilinear Kakeya type estimate, studied in \cite{guth2010endpoint, tao2020sharp, BCT}.
Such an estimate allows us to exploit the randomness of the most contracted directions 
and show ergodicity even when some of the Lyapunov exponents are zero.
As a particular application, we prove that if $(R_1,\ldots,R_m)$ is a tuple of isometries of $S^d$, $d\ge 2$ that generates a dense subgroup of $\SO(d+1)$, then $(R_1,\ldots,R_m)$ is stably ergodic. Previously, this was known only in even dimensions as in this case Lyapunov exponents are non-zero
\cite{dolgopyat2007simultaneous}.

Let $M$ be a connected,  closed smooth  $d$-dimensional 
Riemannian manifold (a closed manifold is a compact manifold without boundary) for an integer $d \geq 2$, and let $\vol$ denote the normalized volume  measure.
Let  $G = \Diff^{\infty}_{\vol}(M)$, the group of $C^{\infty}$ volume preserving diffeomorphisms of $M$.
Consider a compactly supported Borel probability measure $\mu$ on $G$. We study the IID random dynamics on $M$ driven by the measure $\mu$.

Throughout the paper, we let  $\N = \{1, 2, \ldots \}$ and let $\Omega = G^{\N}$ be equipped with probability measure $\mathbb{P} = \mu^{\N}$. We let $\mathbb{E}$ denote the expectation with respect to $\mathbb{P}$.
We will view $(G, \mu)$ as a probability space, and for each $\vartheta \in G$ we denote by $f_{\vartheta}$ the map on $M$.
Given $\omega = (\omega_n)_{n \in \N} \in \Omega$ 
we write $f^n_{\omega}=f_{\omega_n}\circ \cdots \circ f_{\omega_1}$ for its time $n$ dynamics; we denote $f^{k,j}_{\omega}=f_{\omega_{k+j}}\circ \cdots \circ f_{\omega_{k+1}}$.

We say that a Borel subset $A \subset M$ is \emph{$\mu$-almost surely invariant modulo $\vol$-null sets} (or simply $\mu$-almost surely invariant) if for $\mu$-almost every $f \in G$, the symmetric difference $f(A) \Delta A$ has zero volume.
We say that volume on $M$ is \emph{ergodic} if there does not exist any $\mu$-almost surely invariant subset of $M$ of non-trivial volume. 
By the Kakutani ergodic theorem \cite[Theorem 3]{Kakutani}
this is equivalent to the ergodicity of the   measure $\mu^{\N}\otimes \vol$ for the skew product $F\colon \Omega \times M \to \Omega\times M$ given by
$(\omega, x)\mapsto (\sigma \omega, f_{\omega_{1}} x)$, where $\sigma$ is the shift.

 Given a matrix $B$ we denote by $\DS  \con(B)=\min_{\|v\|=1} \|Bv\|$ the {\em conorm} of $B$.
 
 \begin{definition}
 \label{defn:main_hypotheses}
 Let $M$ be a connected, closed smooth Riemannian manifold. 
A probability measure $\mu$ on $\Diff^{1}(M)$ has
 {\em codimension $1$ hyperbolicity gap} if there exists $N\in \N$ such that for any $x\in M$ and any codimension $1$ subspace $E\in T_xM$,
\begin{equation}
  \label{CoDim1Gap}  
\mathbb{E}\!\!\left[\log \con(D_xf^{N}_{\omega}|E)\right]
 -  \mathbb{E}\left[\log\left\| P_{(D_xf^N_{\omega}(E))^{\perp}} D_xf^{N}_{\omega}  u_E \right\| \right ]     > 0,
\end{equation}
where $u_E$ is a unit vector in $E^\perp$ and $P_E$ is the orthogonal projection to $E$.
 
We will denote the set of compactly supported measures on $\Diff^\infty_{\vol}(M)$ enjoying codimension $1$ hyperbolicity gap by $\cG$.
\end{definition}

When the support of $\mu$ is volume preserving, by volume preservation,  we have 
$$\log\left\| P_{(Df^N_{\omega}(E))^{\perp}} D_xf^{N}_{\omega}  u_E \right\|  + \log |\det (D_x f^{N}_{\omega} | E)| = 0$$
for any integer $N \in \N$ and any $x, E$ in Definition \ref{defn:main_hypotheses}.
Thus \eqref{CoDim1Gap} can be rewritten as the following simpler to state condition:   there exists an integer $N > 0$ such that   for all $x, E$ as above
 \begin{equation}
 \label{D-1Gap}  \mathbb{E}\!\!\left[\log \con(D_xf^{N}_{\omega}|E)\right]
 +\mathbb{E}\!\!\left[
 \log |\det( D_xf^{N}_{\omega}| E )| \right]   > 0.
\end{equation}

The main result of this paper is the following.

\begin{theorem}\label{thm:main_thm}
The IID random dynamics driven by a measure $\mu$ from $\cG$ is ergodic for volume.
\end{theorem}

The proof relies crucially on the following property.

\begin{definition}
\label{DefCoA}
Let $M$ be a connected, closed   smooth Riemannian manifold. 
A probability measure $\mu$ on $\Diff^1(M)$
is called 
{\em coexpanding on average} if there exists $N\in \N$ 
    such that for any $x \in M$ and any  $  \xi\!\in\! T^*_xM \setminus \{0\}$, we have
    \begin{equation}
    \frac{1}{N} \mathbb{E}\left(\log \frac{\|(D_x(f^N_{\omega})^*)^{-1}\xi\|}{\|\xi\|}\right)> 0.
    \end{equation}
\end{definition}

Given a separable complete metric space $X$, we denote by $\mathcal{M}(X)$, resp.\ $\mathcal{P}(X)$,  the space of Borel  measures, resp.\ Borel probability measures, on $X$ equipped with the weak-$*$ topology. Then for every bounded open subset $\mathcal{O} \subset \Diff^1(M)$, having codimension $1$ hyperbolicity gap and coexpanding on average are  both open properties in $\mathcal{P}( \mathcal{O} )$.

Definition \ref{DefCoA} comes from \cite{dewitt2025conservative}, where it is shown to imply an essential spectral gap and that volume has finitely many ergodic components.
 By \cite[Prop.\ 3.11]{dewitt2025conservative}, when the dynamics is volume preserving, the coexpanding on average property is actually equivalent to \lq\lq uniformly expanding in dimension $d-1$ \rq\rq given in \cite[Definition 1.1.3]{brown2025measure}.
Definition \ref{defn:main_hypotheses} is a small modification of \cite[Defn.~8]{zhang2019stable}. 
In fact, codimension $1$ hyperbolicity gap is a stronger condition,   as was pointed out  to us by Timoth\'ee B\'enard. The proof of the following lemma appears in Section \ref{sec:discussion_of_hypotheses}.

\begin{lemma}\label{lem:codimension_1_hyperbolicity_gaps_implies_coexpansion}
Every $\mu \in \cG$ is coexpanding on average.
\end{lemma}

From Theorem \ref{thm:main_thm}, Lemma \ref{lem:codimension_1_hyperbolicity_gaps_implies_coexpansion}
and the results of \cite[Section 7]{dewitt2025conservative}
we immediately obtain the following.

\begin{theorem}\label{thm:spectralgap}
    Let $\mu \in \cG$. Then there is $s > 0$ such that the generator $(\cL_\mu \phi)(x)=\int \phi(f(x)) d\mu(f)$ on $H^s(M)$ has a spectral gap. As a result, there exist $C, \kappa > 0$ such that for every $\phi \in H^s(M)$ with $\int \phi d \vol = 0$, we have
    $\DS \| \cL_{\mu}^n\phi  \|_{H^s(M)} < C e^{- n \kappa} \| \phi \|_{H^s(M)}.$
\end{theorem}

We now describe one of our  main applications. Let $d \geq 2$ and let $M = S^d$ (the standard $d$-dimensional sphere). We can naturally identify $\SO(d+1)$ with a subgroup of $G = \Diff^{\infty}_{\vol}(M)$.

\begin{theorem}\label{cor:spheres_corollary}
	Let $\mu_0$ be a Borel probability measure on $\SO(d+1)$,  $d\ge 2$,  such that $\supp(\mu_0)$ generates a dense subgroup of $\SO(d+1)$. 
 Then there is a neighborhood $\cU$ of $\mu_0$ in the space of Borel probability measures on $G$ with respect to the weak-$*$ topology   such that every $\mu \in \cU$  is ergodic for volume.
\end{theorem}

\noindent The proof of Theorem \ref{cor:spheres_corollary} combines the results of \cite{dolgopyat2007simultaneous}
with Theorem \ref{thm:main_thm} 
and is explained
in Section \ref{sec:discussion_of_hypotheses}.

In \cite{dolgopyat2007simultaneous},  Theorem \ref{cor:spheres_corollary} is proved only for even $d$ (for   driving measures of finite support),
since in that case the system has no zero Lyapunov exponents, 
due to an approximate symmetry and simplicity of the Lyapunov spectrum. The argument of
\cite{dolgopyat2007simultaneous} does not extend to the odd dimensional case, see \S \ref{SSPrior} for an additional 
discussion.
Observe that the hypothesis of Definition~\ref{defn:main_hypotheses} only gives \emph{one} negative Lyapunov exponent, which suffices for our arguments.
Note that a zero Lyapunov exponent for the random dynamics can appear quite often: for example, whenever $M$ is of odd dimension and $\mu$ is a symmetric measure.

Once a coexpanding on average system is known to be ergodic, many additional properties like exponential mixing and the central limit theorem follow from \cite{dewitt2025conservative}. In particular, we have the following.

\begin{corollary}
\label{CrSpGap}
Under the conditions of Theorem \ref{cor:spheres_corollary} 
there exist neighborhoods $\mathcal{O}$ of $\SO(d+1)$ in $G$ and 
$\mathcal{U}$
of $\mu_0$ with respect to the weak-$*$ topology on $  \mathcal{P}(\mathcal{O})$, such that for each $\mu\in\mathcal{U}$, 

either (1) There exists a metric $\mathfrak{g}$ on $S^d$ which is preserved by $\mu$ almost every diffeomorphism;

or (2) 
the generator 
$(\cL_\mu \phi)(x)=\int \phi(fx) d\mu(f) $
has a spectral gap on $L^2$.

If the second alternative holds then the system 
enjoys exponential mixing of all orders: for each $r\in \mathbb{N}$ there are constants $C_r>0$ and $\theta_r<1$
such that for all bounded functions 
$\phi_1, \dots ,\phi_r$ for each tuple $n_1, \dots ,n_r$ we have
$$ \left| \mathbb{E} \left[\int \prod_{j=1}^r \phi_j(f_\omega^{n_j} x) d\vol(x) 
-\prod_{j=1}^r \int \phi_j d\vol\right] \right|\leq C_r \prod_{j=1}^r \|\phi_j\|_{L^\infty} \theta_r^L
\quad\text{where} \quad L=\min_{i\neq j} |n_i-n_j|.
$$
Moreover, the Central Limit Theorem holds: if $x$ is uniformly distributed on $M$
 and $\omega$ is distributed
according to $\mathbb{P}$ then for each zero mean $\phi\in L^\infty$
the random variable $S_N/\sqrt{N}$ converges in law  to a normal random variable with zero mean and 
variance $\sigma^2$ where
$$S_N=\sum_{n=0}^{N-1} \phi(f^n_\omega x)\quad\text{and}\quad 
\sigma^2=\int \phi^2 d\vol+2\sum_{n=1}^\infty \mathbb{E} \left[ \int  \phi(x) \phi(f_\omega^n x) d\vol(x)\right].$$
\end{corollary}

It is a well-known open problem whether the spectral gap in $L^2$ also holds in case (1): see
\cite{Gamburd1999-cc, LPSI, LubotzkyWeiss1993,  BourgainGamburd, Fisher-IMRN} and the references therein.

Another application of our main result concerns measure and topological rigidity for random dynamics. The study of rigidity in this setting was pioneered in two breakthrough works: Bourgain–Furman–Lindenstrauss–Mozes \cite{BFLM} for torus linear automorphisms, and Benoist–Quint \cite{BQ1, BQ2, BQ3} for semi-simple Lie group actions on homogeneous spaces. 
The approach of \cite{BQ1}  has been  subsequently generalized  to obtain Teichmüller dynamics rigidity of measures invariant under the upper-triangular subgroup of $SL(2, \R)$ by Eskin, Mirzakhani and Mohammadi \cite{EM, EMM}.
This framework has deeply influenced later developments, including \cite{EL, BRH, Katz, CD, Rod}, culminating in the recent work \cite{brown2025measure}.
Recently, Bénard and He \cite{BH1, BH2} provided an alternative proof of Benoist–Quint's theorem for simple Lie group \cite{BQ1} using techniques closer in spirit to \cite{BFLM}. 
Extending this methodology, \cite{BZ} established measure rigidity and orbit classification for a class of almost conservative random dynamics. 
A key innovation in \cite{BZ} is the ability to handle zero Lyapunov exponents for nonlinear random dynamics. This point  relies crucially on the existence of a spectral gap established in the present paper.

\begin{remark}
Although we have written this paper for measures $\mu$ supported on $\Diff^{\infty}_{\vol}(M)$, Theorem \ref{thm:main_thm} also holds for measures compactly supported on $\Diff^{3}_{\vol}(M)$. While most of the arguments given in the present paper do not require that the maps be more that $C^{1+\text{H\"older}}$, we use \cite[Thm.~1.1]{dewitt2025conservative} to obtain the essential spectral gap.
 The essential spectral gap plays a key role in our argument. 
 \cite{dewitt2025conservative} assumes
that $\mu$ is supported on $C^{\infty}$ diffeomorphisms.
In fact, 
one can obtain the same result as \cite[Thm.~1.1]{dewitt2025conservative} under the weaker assumption that $\mu$ is supported on $\Diff^{1+\alpha}_{\vol}(M)$ \cite{DTZ}. 
We believe that current techniques allow to prove Theorem \ref{thm:main_thm} for $C^2$ systems.
Below $C^2$ many statements need to be strengthened. For example, the multilinear Kakeya estimate we use, Theorem \ref{thm:taos_theorem}, is only stated for $C^2$ curves. 
Theorem \ref{thm:main_thm} is stated for $C^3$ because we rely on Proposition \ref{prop SMT} and the
proof of that proposition given \cite{DDZ2} loses one derivative. While it seems to be possible to improve the
regularity in Proposition \ref{prop SMT} using a more complicated argument we decided to limit our setting to
$C^3$ systems in order to simplify the presentation.
\end{remark}

\subsection{Outline of the proof and paper}

Because  measures from $\cG$ are coexpanding on average, volume has only finitely many totally ergodic components \cite{dewitt2025conservative}
(a stationary measure $\nu$ is totally ergodic if it is ergodic for every power $\mu^q$ of the driving measure $\mu$). So, if one wants to prove ergodicity it suffices to show that there is only one such totally ergodic component.   

The main insight of this paper is a new method to exploit the geometric structure of an arbitrary invariant set $A$. We will first show that a definite portion of $A$ inside of a small ball $B$ is the intersection of finitely many subsets $K_1, \ldots, K_d$ of $A \cap B$,  where each $K_i$ is saturated by long local strong stable manifolds for an infinite itinerary $\omega^i$. Moreover, it can be arranged that for a given word $\omega^i$ all the local strong stable manifolds on this set  point in roughly the same direction, and the directions for the different words are in general position.
The new analytical ingredient is an upper bound for the volume of $B \cap A$ in terms of the volume of $( K_i )_{i = 1}^{d}$.

To see such an estimate in a simplified model, assume for now $d=3$ and $(K_i)_{i=1}^{3}$ are saturated by parallel, straight segments. We may deduce from the Loomis-Whitney inequality \cite{LW} that 
$$\vol(K_1 \cap K_2 \cap K_3)\cdot \vol(B)^{1/2} \leq C (\prod_{i=1}^{3} \vol(K_i))^{1/2}.$$ 
Let $\kappa = \vol(A\cap B) / \vol(B)$, then $\vol(K_1 \cap K_2 \cap K_3) / \vol(B)$ is at most $O(\kappa^{3/2})$, which is much smaller than $\kappa$, unless $\kappa$ is already bounded from below by a positive absolute constant. 
This shows that the set $A$ exhibits an {\it density gap}: the density of $A$ inside 
small balls 
avoids a certain interval in $(0, 1)$.  This clearly contradicts the continuity of the density function if $A$ is not of full volume. Naturally, this picture is simpler than reality in a number of ways.

First, local strong stable manifolds are not straight segments. It turns out that one can replace the Loomis-Whitney inequality by a {\it Multilinear Kakeya Estimate}. 
This belongs to an interesting line of research started from \cite{BCT}, which takes inspiration from both the Restriction Conjecture and the more classical Multilinear Inequalities. The specific estimate we use is a 
curved  multilinear Kakeya estimate
from \cite{tao2020sharp}, which we restate in a more geometric form suitable for our purposes in Section \ref{sec:geometric_lemma_following_guth}. To apply this estimate, additional non-trivial properties of the strong stable manifolds are needed. Whereas one usually obtains the existence of stable manifolds from ergodic theoretic considerations, we actually need to know that at almost every point, with definite probability, the stable manifold exists, is sufficiently regular,
 and is not contained in any given codimension one subspace of the tangent space.
The precise statements are explained in Section \ref{sec:strong_stables} with more technical proofs relegated to
the appendix.

Secondly, we need the fact that invariant sets have many points of high density at all sufficiently small scales.
This relies crucially on the main results of \cite{dewitt2025conservative}.
Namely, under the coexpanding on average assumption above,  the random dynamics is in fact already quite close to being ergodic. 
There exist some $q$ and $\ell$, and a disjoint decomposition $M=U_1\cup\cdots \cup U_\ell$ such that the $q$th power of the random dynamics leaves each $U_{i}$, $1\le i\le \ell$,  invariant. Moreover every power of the $q$th power dynamics on $U_{i}$ is ergodic. Thus to show ergodicity it suffices to show that there is only one such set $U_i$. That paper in fact provides more information about these sets: it shows that $1_{U_i}$ is a function in $H^s$,   for some
$s>0$. This implies that the sets $U_i$ have additional regularity properties, which are crucial for the proof and are explained in Section \ref{sec:geometric_measure_theory}.

We combine the two ingredients above to show that if there is more than one ergodic component then the 
system exhibits a density gap following the outline at the beginning of this subsection. This
is carried out in Section \ref{sec:main_estimate}.

The paper also has one additional section. Section \ref{sec:discussion_of_hypotheses} contains the proofs of the results other than the main theorem stated in the introduction as well as properties of systems with codimension
one hyperbolicity gap, including the proof of coexpansion on average as well as the relation with the results
of \cite{dolgopyat2007simultaneous}. 

\subsection{Relationship with prior works}
\label{SSPrior}

The main contribution of this paper is a new type of ergodicity argument for smooth systems.
For systems admitting a large group of symmetries harmonic analysis or representation theory provide 
powerful tools to establish ergodicity. Beyond such systems, the main tool for studying ergodic properties
comes from hyperbolicity theory. The Hopf argument going back to \cite{Hopf}, which considered surfaces of negative curvature, is the main
method for establishing ergodicity. The Hopf argument
has been extended to uniformly hyperbolic systems \cite{Anosov, AnosovSinai}, partially hyperbolic systems 
\cite{brin1974partially, grayson1994stably,
burns2010ergodicity}, and non-uniformly hyperbolic systems \cite{pesin}, 
hyperbolic systems with singularities \cite{ChernovMarkarian}, among others.

For small perturbations of random rotations studied in \cite{dolgopyat2007simultaneous},
the system is non-uniformly hyperbolic if $d$ is even, so the classical Hopf argument readily gives
ergodicity as there are no zero Lyapunov exponents. In contrast, if $d$ is odd, the system is 
non-uniformly {\em partially} hyperbolic.
For uniformly partially hyperbolic systems, the strongest known ergodicity results are due to 
\cite{burns2010ergodicity}. This result requires strong control on the regularity of invariant foliations
which is far out of reach in the non-uniform setting. Therefore new ideas were needed to handle the odd case.

Much more recently a different approach to ergodicity in partially hyperbolic systems
with zero exponents has been developed \cite{TsujiiFat, TsujiiSE, TsujiiFlow, CL22, TZ26}. 
This method was originally created to prove ergodicity
of systems with singularities \cite{LasotaYorke} and to analyze fine properties
of uniformly hyperbolic systems \cite{GouezelLiverani, BaladiTsujii}
and requires establishing quasi-compactness of the generator by proving an 
appropriate Lasota-Yorke inequality. The spectral approach was also used in
 \cite{tsujii2023virtually} to study endomorphisms and \cite{dewitt2025conservative} to study random systems. 
 Unfortunately the Lasota-Yorke inequality only gives quasi-compactness and as a result 
one only obtains the finiteness of the number of ergodic components.

The key new ingredient of the present paper is the use of Multilinear Kakeya
estimate, a far-reaching  generalization of the classical Loomis-Whitney inequality.
This sharp estimate  has the advantage that 
it requires very little regularity in the transverse direction thus
allowing one to overcome the fractal nature of invariant laminations in the non-uniform setting.
More precisely,  the main analytical estimate used in this paper is Theorem \ref{thm:taos_theorem},
 the (non-endpoint) Curved Multilinear Kakeya Estimate, proved by Tao in \cite{tao2020sharp}.  
 Tao's theorem generalizes the main result in \cite{BCT} under natural conditions, by replacing the straight tubes by tubular neighborhoods of $C^2$ curves. 
 In fact, we only need a weaker version of Tao's theorem, with a logarithmic loss, which can already be deduced from \cite{BCT} and an induction-on-scale argument (see Remark 1.5 in \cite{tao2020sharp}).
 The results in both \cite{BCT} and \cite{tao2020sharp} concern the non-endpoint estimate (namely, $\bp > 1/(d-1)$ in Theorem \ref{thm:taos_theorem}), and both use the heat flow argument. 
 See \cite{bourgain2011bounds} for an application of these estimates to the Restriction Conjecture.

For our dynamical application, it would be more natural to consider the analogous estimate with $\bp = 1/(d-1)$ (the endpoint case).
In fact, in answering a conjecture in \cite{BCT}, the  Endpoint case of Multilinear Kakeya Estimate was proved by Guth in \cite{guth2010endpoint} using the Polynomial Method. It is not hard to adapt the method in \cite{guth2010endpoint} to prove the endpoint estimate when the straight tubes are replaced by the tubular neighborhoods of polynomial curves with uniformly bounded degree (see \cite[comment below Theorem 1.3]{tao2020sharp}). In fact, this would also be   enough for our method provided that the maps are of class $C^{\infty}$.\footnote{We have {\it a priori} $H^s$ regularity of the invariant set. This allows us to take the girth of the tubular neighborhood to be polynomially comparable to the diameter of the ball. Then we may use Taylor expansion to reduce the estimate to the polynomially curved version, provided that curves are of class $C^{\infty}$.}
However, to obtain the endpoint estimate for the $C^2$-curved version as in \cite{tao2020sharp}, it appears that a new idea is required. 
Fortunately, using the a priori $H^s$-regularity of the invariant set, we may allow a mild degeneration of the density gap interval as the scale shrinks.  This turns out to be enough to conclude the proof.

The result of \cite{guth2010endpoint} has been generalized in \cite{zhang2018endpoint}.  It would be interesting to explore the connection between Brascamp-Lieb type inequalities and an analogous result where we assume a higher codimension hyperbolicity gap.

\subsection{Notation}
  
  Given an integer $n > 0$, $z \in \R^n$ and $r > 0$, we denote by $B^n(z, r)$ the closed ball of radius $r$ centered at $z$ in $\R^n$.
  Given $x \in M$ and $r > 0$,  we denote by $B_{r}(x)$  the closed ball of radius $r$ centered at $x$.  
  
  Let $\delta_{inj}$ be the injectivity radius of $M$. In particular, for every $\delta \in (0, \delta_{inj})$ and every $x \in M$, the inverse of the exponential map $\exp_x : T_x M \to M$ is well-defined on $B_{\delta}(x)$.

Given a finite dimensional vector space $V$ and an integer $1 \leq m \leq \dim(V)$, we denote by $\Gr(V, m)$ the Grassmannian that parametrizes the set of $m$-dimensional subspaces of $V$.

\subsection*{Acknowledgements}
The authors thank Artur Avila and Timoth\'ee B\'enard for discussions related to Definition \ref{defn:main_hypotheses} and thank Hong Wang for discussions related to Multilinear Kakeya estimate.
The first author was supported by the National Science Foundation under Award
No.~DMS-2202967. The second author was supported by the National Science Foundation under award
No.~DMS-2246983.

\section{Codimension 1 gap}\label{sec:discussion_of_hypotheses}

In this section, we provide the details for several of the results stated in the introduction. In 
\S\ref{subsec:consequences_of_main_thm}, we prove the dynamical and statistical consequences of our main theorem. In \S \ref{SS1.4}, we show that codimension $1$ hyperbolicity gaps imply coexpansion on average. 

\subsection{Consequences of the main theorem}\label{subsec:consequences_of_main_thm}

In this subsection, we prove Theorem \ref{cor:spheres_corollary},  Corollary~\ref{CrSpGap}, and Lemma \ref{lem:codimension_1_hyperbolicity_gaps_implies_coexpansion}.
Theorem \ref{cor:spheres_corollary} follows quickly from Theorem \ref{thm:main_thm} and the result below.

\begin{proposition} 
\label{PrDKOdd}
 Given an integer $d \geq 2$, there is an integer $\reg_0 > 0$ depending on $d$ such that the following is true.
Let $(R_1,\ldots,R_m)$ be a tuple of isometries of $S^d$  that generates a dense subgroup of $\SO(d+1)$. Then there exists $\eps_0>0$ such that if $(f_1,\ldots,f_m)$ are volume preserving 
$C^\infty$ diffeomorphisms of $S^d$  satisfying
$d_{C^{\reg_0}} (f_j, R_j)\leq \eps_0$ for all $j$, 
then exactly one of the following holds:

\begin{enumerate}
\item[(1)] 
There exists $h\in \Diff^{\infty}(S^d)$ such that $hf_ih^{-1}\in \SO(d+1)$ for all $i$;
    \item[(2)]
    The uniform measure on $(f_1,\ldots,f_m)$ has codimension 1 hyperbolicity gap.
\end{enumerate}
Moreover for each $\delta>0$  we can choose $\eps_0$ so small that in case (1) 
$d_{\SO(d+1)}(h f_i h^{-1}, R_i ) < \delta$.  
\end{proposition}

\begin{proof}
By \cite{dolgopyat2007simultaneous} if the top Lyapunov exponent of the volume $\lambda_1$ is equal to zero,
then alternative (1) holds. (See also \cite{dewitt2024simultaneous}.)
Therefore, we assume that $\lambda_1>0$.

By Corollary 4(a) and Lemma 4 in \cite{dolgopyat2007simultaneous} (applied with $r=d-1$)
for each $\eps>0$ there exists $0\le \theta<1$ such that for $N\ge 1$ with probability
$1-O(\theta^N)$,
\begin{equation}
\label{GapTerms}
    \frac{1}{N} \ln \con(D_xf^{N}_{\omega}|E)>\left(\frac{3-d}{d-1}\right) (\lambda_1-\eps)
    \quad \text{and}\quad
    \frac{1}{N} \log | \det( D_xf^{N}_{\omega}| E ) | >
    \lambda_1-\eps.
\end{equation}
Since both expressions from \eqref{GapTerms} are uniformly bounded from below, taking $\eps$
sufficiently small  we conclude that alternative (2) holds.
\end{proof}

\begin{proof}[Proof of Theorem \ref{cor:spheres_corollary}]
By hypothesis, we may take $R_1, \ldots, R_{m}$ in  $\supp(\mu_0)$ so that $(R_1, \ldots, R_m)$ generates a dense subgroup of $\SO(d+1)$. Let $\reg_0$ and $\eps_0$ be as in Proposition \ref{PrDKOdd} where $\eps_0$
is chosen so small that in case (1) $( hf_i h^{-1} )_{i = 1}^{m}$ generates a dense subgroup.
In this case, any tuple $(f_1, \ldots, f_m)$ satisfying $d_{C^{\reg_0}} (f_j, R_j)\leq \eps_0$
is ergodic. Indeed in case (1), the ergodicity holds since $\SO(d+1)$ acts transitively on the sphere while 
in case (2), it holds due to our Main Theorem \ref{thm:main_thm}.

For each $i$, we let $\cV_i$ be the set of $f_i \in \Diff_{\vol}^{\infty}(S^d)$ with $d_{C^{\reg_0}} (f_i, R_i) < \eps_0$.
Then there is a neighborhood $\cU$ of $\mu_0$ in the weak-$*$ topology on probability measures on $\Diff_{\vol}^{\infty}(S^d)$ such that for every $\mu \in \cU$, we have $\mu(\mathcal{V}_i) > 0$ for all $i$. Then each $\mu\in\mathcal{U}$
is ergodic. Otherwise there would be a set $E$ of intermediate volume which is essentially invariant by
$\mu$-almost every $f$. Since $\mu(\mathcal{V}_i) > 0$ there would be a tuple
$(f_1, \ldots, f_m)$ such that $f_j\in \mathcal{V}_j$ and $f_j$ preserves $E$. But this contradicts
the ergodicity of $(f_1, \ldots, f_m)$ established above. This completes the proof.
 \end{proof}

\begin{proof}[Proof of Corollary \ref{CrSpGap}]
We divide the proof into four steps. After creating the neighborhoods in the statement of the corollary, we will use $\mu$ to define an auxiliary measure $\mc{P}_{\mu}$ on $m+1$ tuples of diffeomorphisms such that any tuple in the support of $\mc{P}_{\mu}$ either has spectral gap or preserves a common Riemannian metric.
Throughout the proof we shall abbreviate $\bd=d_{C^{\reg_0}}$.
\smallskip

{\bf Step I.} If alternative (2) in Proposition \ref{PrDKOdd} holds, then the generator
$\phi\mapsto \mf{L} \phi\coloneqq \frac{1}{m} \sum (\phi\circ f_j)$ has a spectral gap on $L^2.$

\smallskip

Indeed, applying Proposition \ref{PrDKOdd} to both $(R_1, \ldots, R_m)$ and $(R_1^{-1}, \ldots, R_m^{-1})$
we see that the uniform measures on both tuples are coexpanding on average, so 
\cite[Theorem 7.1]{dewitt2025conservative}
shows that $\mf{L}$ has essential spectral gap on $L^2.$ 
On the other hand, if $\mu$ is coexpanding on average then so 
is $\mu^{*n}$ for all $n$ (see the proof of Lemma \ref{lem MomentBound} in the appendix)
so Theorem~\ref{thm:main_thm} shows that our random system
is totally ergodic. Now the spectral gap follows from \cite[\S 7.4]{dewitt2025conservative}.
\medskip

{\bf Step II.} Construction of the $C^{\infty}$ neighborhood $\mc{O}$, the weak* neighborhood $\mc{U}$, and the auxiliary measure $\mc{P}_{\mu}$. 

\smallskip

As in the proof of Theorem \ref{cor:spheres_corollary} above, fix some $R_1, \dots , R_m$ in $\supp(\mu_0)$ that  generate a dense subgroup of $\SO(d+1)$.
Applying Step I to the tuple $(R_1, \ldots , R_m, R)$, we conclude that for each $R\in \SO(d+1)$
there exists $\eps_R>0$ such that if 
$(f_1, \ldots ,f_m, f_{m+1})$ satisfy $\bd(f_j, R_j)\le\eps_R$ for $j=1, \ldots, m$
and $\bd (f_{m+1}, R)\le\eps_R$,  then the alternative of Proposition \ref{PrDKOdd} holds.
By compactness, there exists a finite collection $(\tilde R_p)_{p=1}^k$ of rotations such that the open balls
$B(\tilde R_p, \eps_{\tilde R_p})$ cover $\operatorname{SO}(d+1)$. 
Let $$\DS \mathcal{O}=\bigcup_{p=1}^k \{f| \bd(f, \tilde R_p)< \eps_{\tilde R_p}\}, \quad
\mathcal{V}_j=\{f| \bd(f, R_j)\leq \min_p \eps_{\tilde R_p}\} , \quad
\mathcal{U}=\{\mu| \mu(\mathcal{V}_j)>0 \text{ for } j=1, \ldots , m\}.$$ 
We will show that the conclusion of the corollary holds with $\mathcal{O}$ and $\mathcal{U}$
defined above. Let $\mu$ be such a measure, let $\mu_j$ denote the measure $\mu$ conditioned on $\mathcal{V}_j$, and consider the measure on 
$G^{m+1}$ of the form
$\mathcal{P}_{\mu}=\mu_1\times \cdots \times \mu_m\times \mu$. 
By the foregoing discussion $\mathcal{P}_{\mu}\left(\mathbb{I}\bigcup\mathbb{G}\right)=1$
where $\mathbb{I}$ is the set of tuples $(f_1, \ldots , f_m, f)$ preserving a  Riemannian metric
on $S^d$ and $\mathbb{G}$ is the set of tuples $(f_1,\ldots,f_{m},f)$ that have
dynamics with  coexpansion and $L^2$ spectral gap.
\medskip

{\bf Step III.} If $\mathcal{P}_{\mu}(\mathbb{I})\! = \! 1$, then there exists a metric $\mathfrak{g}$ on $S^d$
that is preserved by $\mu$-almost every $f.$ 

\smallskip

Indeed, if $\mathcal{P}_{\mu}(\mathbb{I})\!\!=\!\! 1$, then by Fubini there exists 
$(f_1, \ldots, f_m)$ with $\bd(f_j, R_j)\leq \eps_0$ such that 
for $\mu$-almost every $f$ the tuple $(f_1, \ldots, f_m, f)$ preserves some unit volume  metric $\mathfrak{g}_f$
on $S^d$. 
Since by our choice of $\eps_0$, $(f_1, \ldots  , f_m)$ generate a dense subgroup of Isom$(\mathfrak{g}_f)$,
$\mathfrak{g}_f$ is the unique unit volume  metric that is preserved by $(f_1, \ldots , f_m)$. It follows that 
$\mathfrak{g}_f$ does not in fact depend on $f$ and so the first alternative of the corollary holds.

\medskip

{\bf Step IV.} If $\mathcal{P}_{\mu}(\mathbb{G})>0$, then the generator $\mc{L}_{\mu}$ has spectral gap on $L^2$. 

\smallskip

Indeed, in this case there are distinct diffeomorphisms $\cF=\{f_1, \ldots, f_{m+1}\}\subset \supp(\mu)$ such that
the uniform dynamics on $\cF$ has a spectral gap. Applying Theorem 7.6 from 
\cite{dewitt2025conservative}, we conclude that there is $\eps_1>0$ such that if $\tilde\mu$
is any measure with 
\begin{equation}
\label{UniBlob}
  \tilde\mu\{f| \bd(f, f_j)<\eps_1\}=\frac{1}{m+1},  
\end{equation}
then the generator $\cL_{\tilde\mu}$ has a spectral gap on $L^2.$ Since $f_j$ are in $\supp(\mu)$
it follows that $\mu$ can be decomposed as $p\tilde\mu+(1-p)\nu$ where $p > 0$ and $\tilde \mu$ satisfies
\eqref{UniBlob}.
Accordingly there exists $N\in \mathbb{N}$ and $\theta>0$ such that for such $\wt{\mu}$, 
$\|\cL_{\widetilde{\mu}}^N\|_{L^2_0}\leq 1-\theta$ where $L^2_0$ stands for the mean-zero subspace of $L^2$.
Thus
$\cL_\mu^N=p^N\cL_{\widetilde{\mu}}^N+(1-p^N)\cL_{\nu_N}$ where $\nu_N$ describes the contribution
of remaining terms. Since the composition with a volume preserving diffeomorphism is an isometry on $L^2$,
$\|\cL_{\nu_N}\|_{L^2}\leq 1$. Step IV follows.

\medskip

The dichotomy claimed in the corollary follows from Steps III and IV above.
The other statements follow from the spectral gap on $L^2$, 
see  \cite[\S7.3 and \S7.8]{dewitt2025conservative} for details.
\end{proof}

\subsection{Codimension one hyperbolicity gaps imply coexpansion}
\label{SS1.4}

\begin{proof}[Proof of Lemma \ref{lem:codimension_1_hyperbolicity_gaps_implies_coexpansion}]
	
	From \cite[Prop.\ 3.11]{dewitt2025conservative}, coexpanding on average is the same thing as being expanding on average on $(d-1)$-planes, so this is what we will check. 
	
	Suppose that $N$ is as in the codimension $1$ hyperbolicity gaps condition. Then
	\[
	\mathbb{E}\!\!\left[\log\con (Df^N_\omega| E)\right]
	+
	\mathbb{E}\!\!\left[\log | \det (D_x f_{\omega}^N | E )|
     \right] > 0.
	\]
	We also have the trivial bound
	$\DS
		\log  | \det (D_x f_{\omega}^N | E ) | \ge (d-1) \log \con (D_xf^N_\omega|E)
	$. Hence,
	$$
	d	\log  | \det (D_x f_{\omega}^N | E ) |
	\ge (d-1)	\left[ \log  | \det (D_x f_{\omega}^N | E ) | +  \log \con (Df^N_\omega| E) \right].
	$$
	Taking the expectation on both sides, and using the codimension-1 gaps, we  conclude that $\mu$ is coexpanding on average.
\end{proof}

\section{Strong stable manifolds and laminations}\label{sec:strong_stables}

The starting point of our discussion is Oseledec's Theorem for random dynamics, 
which we now recall.
 Given a Borel probability measure $\mu$ compactly supported on $G = \Diff^{\infty}_{\vol}(M)$,  the shift map $T \colon  \Omega \to \Omega$, $T (\omega_n)_{n \geq 1} = (\omega_{n+1})_{n \geq 1}$ preserves $\PP = \mu^{\otimes \N}$; and the skew product
 $$
 F : \Omega \times M \to \Omega \times M  \quad\quad
 (\omega, x) \mapsto (T(\omega), f^1_{\omega}(x))
 $$
 preserves $\PP \otimes \vol$.
The following proposition is from  \cite[Chapter I, Theorem~3.2]{Liu1995smooth}.
 \begin{proposition} \label{prop randomMET}
There is a function $r
\colon M \to \{1, \ldots, d\}$ and 
 a Borel subset $\Lambda_0 \subset \Omega \times M$ with $\PP \otimes \vol(\Lambda_0) = 1$ such that $F(\Lambda_0) = \Lambda_0$, and   for every $(\omega, x) \in \Lambda_0$ there exist a sequence of linear subspaces of $T_x M$
\aryst
\{0\} = V^{(0)}(\omega, x) \subsetneq V^{(1)}(\omega, x) \subsetneq \cdots \subsetneq V^{(r(x))}(\omega, x) = T_x M,
\earyst
with $\dim V^{(i)}(\omega, x)$ depending only on $x$, $i$, 
 and numbers 
\aryst
\lambda^{(1)}(x)   < \cdots < \lambda^{(r(x))}(x)
\earyst
such that for all $\xi \in V^{(i)}(\omega, x) \setminus V^{(i-1)}(\omega, x)$, $1 \leq i \leq r(x)$
\aryst
\lim_{n \to \infty} \frac{1}{n} \log \|  D_x f_{\omega}^{n}(\xi) \| = \lambda^{(i)}(x).
\earyst
 
 Moreover, $r(x),  \lambda^{(i)}(x), V^{(i)}(\omega, x)$ depend measurably on $(\omega, x) \in \Lambda_0$, and for every $(\omega, x) \in \Lambda_0$ and $1 \leq i \leq r(x)$,
$$
r( f^{1}_{\omega}(x) ) = r(x), \  \lambda^{(i)}( f^{1}_{\omega}(x) ) = \lambda^{(i)}(x), \
D_x f^{1}_{\omega}( V^{(i)}(\omega, x) ) = V^{(i)}( F(\omega,  x) ).
$$
 \end{proposition}
 
We fix some $\mu \in \cG$ for the rest of this paper. 

\begin{definition}\label{defn:subtempered_splitting}
Given $(\omega, x) \in \Lambda_0$,
we say that $(D_xf^n_
{\omega})_{n\ge 0}$ has a $(\fC,\kappa,\epsilon)$\emph{-codimension $1$ subtempered dominated splitting} if there exists a pair of subspaces $V\oplus W=T_xM$ with $\dim V=1$ such that 
 denoting $W_k=D_x f_\omega^k (W)$, $V_k=D_x f^k_\omega( V)$, $k \in \N$, we have
\begin{equation}\label{eqn:subtempered_splitting_norm}
  \|D_{f^k_{\omega}(x)}f^{k,j}_{\omega}|V_k\|\le e^\fC e^{-\kappa j}e^{k\epsilon}\con(D_{f^k_{\omega}(x)}f_\omega^{k,j}|W_k),
\end{equation}
and
\begin{equation}\label{eqn:subtempered_splitting_angle}
\angle({W_k, V_k})\ge e^{-\fC}e^{-\epsilon k}.
\end{equation}
We additionally say that this trajectory has $(\fC',\kappa',\epsilon')$\emph{-strong stable contraction} if 
\begin{equation} \label{eq itm3}
\|D_{f_{\omega}^k(x)}f_{\omega}^{k,j}\vert  V_k \|\le e^{-\kappa' j+\epsilon'k+\fC'}.
\end{equation}
If both notions hold with $\fC=\fC'$, $\epsilon=\epsilon'$, we say that $(D_xf^n_{\omega})_{n\ge 0}$ has a $(\fC,\kappa,\kappa',\epsilon)$\emph{-codimension $1$ strong stable splitting}.
\end{definition}

The following is proved in Appendix \ref{app:construction_of_splitting}.

\begin{proposition}\label{prop:tail_on_splitting}
Given $\mu \in \cG$ 
there exist $\kappa,\kappa' > 0$ such that for all $\epsilon>0$, there exist $R,\eta > 0$ such that for $\vol$-a.e. $x\in M$ and $\fC>0$,\;\;
$\DS
\mathbb{P}(\mathcal{A})\!
\ge\! 1\!-\! R e^{-\eta \fC}
$
 where $\mathcal{A}$ is the event that $( D_xf^n_{\omega} )_{n \geq 0}$
admits $(\fC,\kappa,\kappa',\epsilon)$-codimension 1 strong stable splitting with $V=V^{(1)}(\omega, x)$ (see Proposition \ref{prop randomMET}).
\end{proposition}

By Proposition \ref{prop:tail_on_splitting},  $\dim V^{(1)}(\omega, x) = 1$ in Proposition \ref{prop randomMET}
for $\vol$-a.e.~$x \in M$ and $\PP$-a.e.~$\omega$. In the following, for a.e.~$(\omega, x) \in \Lambda_0$ we denote 
by $\DS
E^{ss}(\omega, x) = V^{(1)}(\omega, x)
$,
 the $1$-dimensional subspace of $T_xM$ associated to the Lyapunov exponent $\lambda^{(1)}(x)$, and refer to $E^{ss}(\omega, x)$ as the \emph{$\omega$-strong stable direction at $x$}.
 
 In the rest of the paper, we fix $\kappa, \kappa'$ given by Proposition \ref{prop:tail_on_splitting}.
  Fix $\bar\epsilon < \min(\kappa, \kappa', \epsilon_0) /  10$ where
 $\epsilon_0>0$ comes 
 from Proposition \ref{prop SMT} below.
 Given $\omega \in \Omega$ and $\fC > 0$, we denote by $\cR(\omega, \fC)$ the set of $x \in M$ such that $(D_x f_{\omega}^n)_{n \geq 0}$ has a 
 $(\fC,\kappa,\kappa',  {\bar\epsilon})$-codimension $1$ strong stable splitting. The set $\cR(\omega, \fC)$  can be  thought of as a  {\it Pesin set} for the itinerary $\omega$.

Now we summarize the properties of stable manifolds and laminations needed for this work. The next result follows from  \cite[Chapter III, Theorem 3.2]{Liu1995smooth}.   

\begin{proposition}
By replacing $\Lambda_0$ in Proposition \ref{prop randomMET} by a smaller conull set if necessary,  for every $(\omega, x) \in \Lambda_0$, the set 
	\aryst
	W^{ss}(\omega, x) = \{ y \in M \mid  \limsup_{n \to \infty} \frac{1}{n} \log d(f_{\omega}^n(x), f_{\omega}^n(y)) \leq \lambda^{(1)}(x) \}
	\earyst
	is the image of $E^{ss}(\omega, x)$ under an injective $C^\infty$ immersion satisfying  $T_x W^{ss}(\omega, x)  =  E^{ss}(\omega, x)$.  
  \end{proposition}

  It is a standard fact that the strong stable manifolds are as regular as the dynamics,  see for example, \cite[Thm.\ 7.3.19]{arnold1998random} or \cite[Thm.\ 1]{barreira2008regularity}.

In particular, $W^{ss}(\omega, x)$ is defined for $\PP \otimes \vol$-a.e.\ $(\omega, x)$. We will refer to $W^{ss}(\omega, x)$ as the \emph{$\omega$-strong stable manifold at $x$}.  
We denote by $m_x^s$ the leafwise Lebesgue measure on $W^{ss}(\omega, x)$ (we will suppress $\omega$ from the notation as $\omega$ is often clear from the context).
We will also use the following notation: Let $U$ be an open set, we denote by $(W^{ss}(\omega, x) \cap U)_{x}$ the connected component of $W^{ss}(\omega, x) \cap U$ containing $x$.

We are interested in strong stable manifolds because every $\mu$-almost surely invariant set  is  almost saturated by strong stable manifolds. 

\begin{theorem} \label{thm ergodiccomponent}
	Let $A$ be a $\mu$-almost surely invariant subset of $M$ with $\vol(A) > 0$. Then  for $\vol$-a.e. $x \in A$, for $\PP$-a.e. $\omega$,  $m^s_{x}$-almost every $y \in W^{ss}(\omega,x )$ belongs to $A$.
\end{theorem}

 It is important for our method to have quantitative control of the strong stable manifolds. The following notion is useful. 

\begin{definition}
Let $\gamma\subset M$ be a curve and $x\in \gamma.$ We say that $\gamma$ is {\em $(\fC, \reg)$-good at $x$} if there is a subcurve $x\in \hat\gamma\subset \gamma$ such that
$\exp^{-1}_x  \hat\gamma$ is the graph of a function $\phi_x\colon  \tilde V\to V^\perp$ with 
$\|\phi_x\|_{C^\reg}\leq \fC$
where $V$ is the tangent direction to $\gamma$ at $x$, and $\tilde V$ is the ball of radius $1/\fC$ in $V$ centered at 0.

Given a ball $B\ni x$ and $\fC$, $\gamma$ as above, we say that $\gamma$    $\fC$-properly crosses   $B$ if
$\gamma \cap B$ is a connected subcurve of $\gamma$ of   length  at least $r/\fC$ where $r$ is the radius of $B$.
\end{definition}

 In this paper we will only use this property with $\reg=2$,  therefore below we write $\fC$-good for
 $(\fC, 2)$-good.  
 
A standard way for controlling the geometry of stable manifolds and constructing stably laminated sets is via non-uniform hyperbolicity.

\begin{proposition} \label{prop SMT}
Let   $\reg \geq  1$, and suppose that $\mu$ is a compactly supported measure on $\Diff^{\reg+1}_{\vol}(M)$. For  any $\kappa,\kappa'>0$, there exists $\epsilon_0>0$ such that for any $0<\epsilon<\epsilon_0$, for any $\fC > 0$, there exists $D > 0$ such that 
    if for some  $(\omega, x) \in \Lambda_0$, \; $(D_xf^n_{\omega})_{n\ge 0}$ has a $(\fC,\kappa,\kappa',\epsilon)$-codimension $1$ strong stable splitting,  then  $W^{ss}(\omega, x)$ is $(D, \reg)$-good   at $x.$
\end{proposition}

\begin{proof}   
This result follows from the usual construction of the stable manifolds \cite{barreira2007nonuniform}, \cite{shub1987global}. In fact assumptions in Pesin theory are slightly different, since it is commonly assumed that vectors from $V$ grow with at most fixed rate $\lambda^{(1)}+\eps$ and the vectors
from $W$ grow with at least fixed rate $\lambda^{(2)}-\eps$ with $\lambda^{(2)}>\lambda^{(1)}.$
However, the proofs in the  references mentioned above  only use the domination 
in the sense of \eqref{eqn:subtempered_splitting_norm}, so they go through under our assumption with minor
modifications. For details, see the manuscript \cite[Thm.~1.3,Cor.~1.5]{DDZ2}.
 \end{proof}
 
 Just like for deterministic iterations, it is known that for a $\PP$-typical $\omega$, the collection of $\omega$-strong stable manifolds forms an absolutely continuous measurable lamination, and the conditional measure along the strong stable leaves restricted to some $\cR(\omega, \fC)$ has bounded density.  
 We now explain this point in detail.
 	In the course of the proof we will study the strong stable laminations restricted to a ball $B_{\delta}(x)$ for some $x\in M$ and a small $\delta>0$. 
 	Given $\omega \in \Omega$ and $\fC > 0$,  we let 
 	\[
 	\Lambda^{ss}(\omega, \fC,  x,\delta) = \bigcup_{y\in \cR(\omega, \fC) \cap B_{\delta}(x)}  \Big (W^{ss}(\omega, y) \cap B_{\delta}(x) \Big )_y. 
 	\]
    When $\delta$ is sufficiently small depending on $\fC$, the right hand side of the above equation forms a measurable partition of $\Lambda^{ss}(\omega, \fC, x,\delta)$ into a set of local strong stable manifolds.
 	By Rokhlin's Disintegration Theorem we have a decomposition
 	\begin{align} \label{eq decomposition}
 		\vol|_{\Lambda^{ss}(\omega, \fC, x,\delta)} = \int_{\Lambda^{ss}(\omega, \fC, x,\delta)} \nu_y^s 
         d\vol(y)
 	\end{align}
   where $\nu_y^s$ is a probability measure supported on $W^{ss}(\omega, y)$ for $\vol$-a.e.~$y \in	\Lambda^{ss}(\omega, \fC,  x,\delta)$.
    Moreover, for $\vol$-a.e. $y \in \Lambda^{ss}(\omega, \fC, x,\delta)$, for $\nu_y^s$-a.e.\ $z$, we have 
    $\nu_y^s = \nu_z^s $.
 	 	For simplicity, we are  suppressing the dependence on $\kappa,\kappa',\epsilon$ in  these notations. 
 	
 	   The following statement could be proven by the arguments from the proof of 
     \cite[Chapter III, Theorem 6.1]{Liu1995smooth}.  As mentioned in the proof of Proposition \ref{prop SMT}, we do not make assumptions on either contraction rate on $V$ or minimal expansion of $W$, but the
     proof from the above reference applies without significant changes. Alternatively we can directly deduce
     Proposition \ref{prop Conditional} from absolute continuity of the (weak) stable lamination which
     follows from \cite[Chapter III, Theorem 6.1]{Liu1995smooth} and the smoothness of strong stable lamination inside weak stable leaves proven in \cite{brown2022smoothness}. See also 
\cite[\S 8.6]{barreira2007nonuniform}. Although the argument is standard, see \cite[\S 7]{DDZ2}, where a detailed argument is given that applies in our specific setting.
 	\begin{proposition} \label{prop Conditional} 
 		Suppose that $\mu$ is a compactly supported measure on $\Diff^{2}_{\vol}(M)$. 
 		For every $\fC > 0$,  there exist $\delta_0>0$ and $\hat{C}>0$ such that the following holds. For all $x\in M$, $0<\delta\le \delta_0$, and $\PP$-a.e.\ $\omega$, for $\vol$-a.e.\ $y \in 	\Lambda^{ss}(\omega, \fC,  x,\delta)$, $\nu_y^s$  is absolutely continuous with respect to   $m_y^s$
        with density $\varrho_y$ supported on $( W^{ss}(\omega, y)\cap B_{\delta}(x) )_y$, satisfying that for $m_y^s$-a.e. $z,z'\in ( W^{ss}(\omega, y)\cap B_{\delta}(x) )_y$
 		\begin{equation}
 		    \label{DensityComp}
1/{\hat C}<   \varrho_y(z)/\varrho_y(z') \le \hat C. 
 		\end{equation}
 	\end{proposition}

  Another important ingredient in our proof is that the stable direction is not deterministic. The following lemma is proved in Appendix \ref{app:construction_of_splitting}.

\begin{lemma}\label{lem:ahlfors_Ess}
Suppose that $\mu\in \cG$. Then there exist $C',\alpha>0$ such that  for $\vol$-a.e.\ $x\in M$, any  $W \in \Gr(T_xM, d-1)$, and any $\rho > 0$,
\begin{equation}
\mathbb{P}(\omega \mid \angle(E^{ss}(\omega,x), W)<\rho)<C' \rho^{\alpha}.
\end{equation}
\end{lemma}

\begin{corollary}\label{lem:pesin_blocks1}
	Suppose that $  \mu\in \cG$. 
	For every $\eps > 0$, there exist $\fC > 0$ and $\delta_1 \in (0, \delta_{inj})$ such that for every $x \in M$, for every $\delta \in (0, \delta_1)$, for $\vol$-a.e.\ $y \in B_{\delta}(x)$, we have 
	\begin{align*}
 \PP( \omega \mid \mbox{$y \in  \cR(\omega, \fC)$ and   $W^{ss}(\omega, y)$ $\fC$-properly crosses $B(x, \delta)$}) > 1 - \eps.
	\end{align*}
\end{corollary}
\begin{proof}   
	By Proposition \ref{prop:tail_on_splitting},  given $\eps$ there exists $\fC > 0$, such that   for $\vol$-a.e.\ $y \in M$, with probability at least $1-\eps/2$, $y \in \cR(\omega, \fC)$. 
	By Proposition \ref{prop SMT} (for $\reg = 2$), there is $\fD > 0$ such that  $W^{ss}(\omega, y)$ is $\fD$-good at $y$ once $y \in \cR(\omega, \fC)$.
    We let $\rho \in (0, 1/2)$ be a parameter much smaller than $(\eps/(2C'))^{1/\alpha}$ where $C',  \alpha$ are given by  Lemma \ref{lem:ahlfors_Ess}.

Now we fix an arbitrary $x \in M$.
We let $\delta \in (0, 1/\fD)$ be a small parameter  so that the exponential charts at $y$ for all $y \in B_{\delta}(x)$ are $\rho^2$-close. Take a $\vol$-typical $y \in B_{\delta}(x) \setminus \{x\}$. 
	Applying Lemma~\ref{lem:ahlfors_Ess}  for $W = (y-x)^{\perp}$ (in the exponential chart at $x$) with such $\rho$,  we see that with probability at least $1 - \eps/2$, $\angle( E^{ss}(\omega, y), W) > \rho$. Leaving $\fD$ much larger than $\rho^{-1}$, we see that $W^{ss}(\omega, y)$ 
    $\fD$-properly crosses $B_{\delta}(x)$ as long as $W^{ss}(\omega, y)$ is $\fD$-good at $y$ and $\angle( E^{ss}(\omega, y), W) > \rho$ provided $\delta$ is small enough. We conclude the proof by setting $\fC > \fD$.
\end{proof}

In a modest abuse of notation, if $E_1,\ldots,E_d$ are one dimensional subspaces of $V$, where $V$ is an inner product space of dimension $d$, we will write $\abs{\det(E_1,\ldots,E_d)}$ for $\abs{\det(e_1,\ldots,e_d)}$ 
where for $1\le i\le d$, $e_i$ is a unit vector in $E_i$.

\begin{corollary}[Pointwise non-degeneracy] \label{cor:pw_nondegeneracy}
    Suppose that $\mu \in \cG$. Then there exists $\theta_1>0$   such that for $\vol$-a.e.\ $x\in M$ we have 
    \[
\mathbb{P}^{\otimes d}\Big ((\omega^1, \ldots, \omega^d) \in \Omega^{d} \mid  \abs{\det(E^{ss}(\omega^1,x),\ldots,E^{ss}(\omega^d,x))}>\theta_1 \Big )>\theta_1. 
    \]
    \end{corollary}
\begin{proof}  
The proof is by induction. Fix an arbitrary $x \in M$ so that Lemma \ref{lem:ahlfors_Ess} applies.
We pick an arbitrary $\omega^1 \in \Omega$ so that $E^{ss}(\omega^1, x)$ is defined, and pick a unit $v_{1} \in E^{ss}(\omega^1, x)$. 
Assume that for $1 \leq j \leq d-1$ there is some $\beta_j > 0$ depending only on $\mu$ and a  set of $(\omega^1, \ldots, \omega^j)$ with $\PP^{\otimes j}$-measure at least $2^{-j}$ such that   for  unit vectors $( v_{k} \in E^{ss}(\omega^k, x) )_{k = 1}^{j}$ we have     $\| v_{1} \wedge \cdots \wedge v_{j} \| > \beta_{j}$. Pick some  $W \in \Gr(T_{x}M, d-1)$ containing $E^{ss}(\omega^k, x)$ for every $k \in \{ 1, \dots, j \}$. Then by Lemma \ref{lem:ahlfors_Ess},
$\DS\mathbb{P}(
\angle(E^{ss}(\omega^{j+1}, x), W) > \theta')\geq 1/2$
for some $\theta' > 0$ depending only on $\mu$.
  If this event happens then
$\| v_{1} \wedge \cdots \wedge v_{j+1} \| > \beta_{j+1}$ for some $\beta_{j+1} > 0$ depending only on $\beta_{j}$ and $\mu$. The proof is completed when we arrive at $j = d$.
\end{proof}

In the following, we will find for each $x \in M$ a typical itinerary $\omega \in \Omega$ so that a definite portion of a small ball centered at $x$ is laminated by 
$\omega$-strong stable manifolds in $B_{\delta}(x)$ which are sufficiently smooth and almost parallel.

\begin{definition} \label{def goodandaligned}
	Given $\fC > 0$, $\delta \in (0, \delta_{inj})$, $x \in M$, 
	and $\omega \in \Omega$ we say that a Borel subset  $K \subset B_{\delta}(x)$ is \emph{$(\omega, \fC)$-stably laminated} 
if there is a Borel subset $S \subset B_{\delta}(x)   \cap \cR(\omega, \fC)$ such that 
$$K = \bigcup_{y \in S} \Big (W^{ss}(\omega, y) \cap B_{\delta}(x) \Big )_y,$$
and for every $y \in S$, $W^{ss}(\omega, y)$ is $\fC$-good  at $y$ and $\fC$-properly crosses $B_{\delta}(x).$ 

Given $K$ as above, $v\in T_x M\setminus \{0\}$ and $\eps>0$ we say that $K$ is {\em $\eps$-aligned with $v$} 
if  for every $z \in K$, the angle between $E^{ss}(\omega, z)$ and $v$ (viewed in the exponential chart at $x$) is at most $\eps$. 
\end{definition}

The main result of this section is the following.

\begin{proposition} \label{prop:laminations_exist}
	Suppose that $\mu \in \cG$. Then there exist $\fD>0, \theta_0>0$ such that for any $\eps >0$ there exist $\delta_0 \in (0, \delta_{inj})$, $\eta_0 >0$ such that for any $x\in M$,  any $0<\delta<\delta_0$ and   any Borel subset $A\subseteq B_{\delta}(x)$ with $\vol(A) > 0$,  there exist ${\bf v} = (v_1,\ldots,v_d )\in ( T^1_xM )^d$ with $\abs{\det(v_1,\ldots,v_d)}>\theta_0$, and a subset $\Omega_0 = \Omega_0(x, \delta, A, {\bf v}, \eps ) \subset \Omega^{d}$ with $\PP^{\otimes d}(\Omega_0) > \eta_0$ such that for every $(\omega^1, \ldots, \omega^d) \in \Omega_0$:
	\begin{enumerate}[leftmargin=*]
		\item \label{itm enoughlaminations}  
		for every $1 \leq i \leq d$, there is an  $(\omega^i, \fD)$--stably laminated  Borel subset 
        $K_i \subset B_{\delta}(x)$   
    that is  $\eps$-aligned with $v_i$, 
		\item \label{itm coverbigmeasure} $\DS \vol\left(A \bigcap \bigcap_{i = 1}^{d} K_i\right) > \eta_0 \vol(A)$. 
	\end{enumerate}
\end{proposition}

\begin{proof}
Let $\theta_1$ be given by Corollary \ref{cor:pw_nondegeneracy}.
By Corollary \ref{lem:pesin_blocks1} (with $\theta_1/(2d)$ in place of $\eps$),  there is $\fC > 0$ such that for $\vol$-a.e.\ $y$,
\begin{align} 
 \label{eq PPd>1-theta/2} 
 \PP^{\otimes d}\Big ( (\omega^1, \cdots, \omega^d) | &  y  \in  \cR(\omega^i, \fC) \mbox{ and  
 $W^{ss}(\omega^i, y)$ $\fC$-properly   } 
 	\mbox{crosses $B(x, \delta)$ $\forall 1 \leq i \leq d$} \Big ) \!>\! 1 \!-\! \frac{\theta_1}{2}. 
\end{align}

	By Proposition \ref{prop SMT}, there is $\fD \geq \fC$ such that  $W^{ss}(\omega, y)$ is $\fD$-good at $y$ once $y \in \cR(\omega, \fC)$.
Given $\eps > 0$, by letting $\delta_0$ be sufficiently small (depending on $\fD$ and $\eps$), the following is clear:   for every $x \in M$, for $\PP$-a.e.~$\omega$, for $\vol$-a.e.~$y \in B_{\delta}(x)$ such that $W^{ss}(\omega, y)$ is $\fD$-good at $y$, we automatically have that 
$W^{ss}(\omega, y)$ is $\eps/2$-aligned with $E^{ss}(\omega, y)$ in $B_{\delta}(x)$ (see Definition \ref{def goodandaligned}).
From now on we let $\delta \in (0, \delta_0)$ and consider $A \subset B_{\delta}(x)$ with $\vol (A) > 0$.
We will measure the angle between different vectors using the exponential chart at $x$.

For  $y \in A$, we denote by $\wt{\Omega}_y$ the set of $(\omega^1, \ldots, \omega^d) \in \Omega^d$ such that \smallskip

$\abs{\det(E^{ss}(\omega^1,y),\ldots,E^{ss}(\omega^d,y))}>\theta_1$ and   for every $i \in \{1, \ldots, d\}$,  
$y  \in  \cR(\omega^i, \fC)$ and   $W^{ss}(\omega^i, y)$ is $\fD$-good at $y$, $\fD$-properly  crosses 
$B(x, \delta)$, and is  $\eps/2$-aligned with $E^{ss}(\omega^i, y)$ in $B_{\delta}(x)$. \smallskip

By  \eqref{eq PPd>1-theta/2}, Corollary \ref{cor:pw_nondegeneracy},  and the discussion above,  for $\vol$-a.e.~$y \in A$,
\ary \label{eq PP1}
  \PP^{\otimes d}(\wt{\Omega}_y)  > \theta_1 /2. 
\eary
 Denote 
$\DS 
Q\!\!:=\!\!\{(y,\omega^1,\ldots,\omega^d) \!\!\in\!\! A \times \Omega^{d} | (\omega^1,\ldots,\omega^d)\in \wt{\Omega}_y \}.
$
By \eqref{eq PP1},  $(\vol \otimes \PP^{\otimes d})(Q) > \theta_1 \vol(A) / 2$.

Let $H(\theta_1)$ be the subset ${\bf v} \!\!=\!\! (v_1, \ldots, v_d) \in (T_x^1 M)^d$ such that   
$\abs{\det(v_1,\ldots,v_d)}\!\!\geq\!\! \theta_1$. 
Let $J$ be a $\eps/2$-net in $H(\theta_1)$ with respect to the Riemannian metric on $ (T_x^1 M)^d$. 
 Then we have a map $j\colon Q\to J$ assigning $(E^{ss}(\omega^1, y),\ldots,E^{ss}(\omega^d, y))$ to an element ${\bf v} = (v_1, \ldots, v_d) \in J$ such that $\DS \max_{1 \leq i \leq d} \angle(v_i, E^{ss}(\omega^i, y)) < \eps/2$. 
 By the Pigeonhole Principle,  there is ${\bf v} = (v_1, \ldots, v_d) \in J$ such that 
 $$ ( \vol \otimes \PP^{\otimes d} ) ( j^{-1}({\bf v}) ) >  \vol (A) \theta_1 / (2 |J|).$$
 Let $\eta_0 = \theta_1/(4|J|)$.
 Then there is a subset $\Omega_0 \subset \Omega^d$ such that $\PP^{\otimes d} (\Omega_0) >  \eta_0$ and for each $\hat\omega = (\omega^1, \ldots, \omega^d) \in \Omega_0$, we have $\vol(A_{\hat\omega}) > \eta_0 \vol(A)$ where
 \[
A_{\hat\omega}  = \{  y \in A \mid  (y, \omega^1, \ldots, \omega^d) \in j^{-1}({\bf v}) \} \subset 
\bigcap_{i = 1}^{d} \cR(\omega^i, \fC).
 \]
 For $\hat\omega \in \Omega_0$ and $1 \leq i \leq d$, let
 $\DS
 K_i = \bigcup_{y \in A_{\hat\omega}} ( W^{ss}(\omega^i, y) \bigcap B_{\delta}(x) )_y.
 $
 
 Item \eqref{itm enoughlaminations} follows immediately from the construction and Definition \ref{def goodandaligned}.
 Item \eqref{itm coverbigmeasure} holds because 
 $\DS A_{\hat\omega} \subset \bigcap_{i = 1}^{d} K_i$ and $\vol(A_{\hat\omega}) > \eta_0 \vol(A)$.
\end{proof}

\section{Some facts about Sobolev spaces}\label{sec:geometric_measure_theory}

We need some facts about the sets whose indicator is in $H^s$. We begin with the case of functions on $\R^d$, which is a bit simpler and already contains the main points. See \cite[Ch.I]{triebel1992theory} for a historical overview that encompasses everything that we mention below. See also \cite{triebel2010theory}. 
For small $s>0$, the $H^s$ norm in $\R^d$ is equivalent to the following norm
(\cite[Thm.~7.47]{adams2003sobolev}):
\begin{equation}\label{eqn:Hs_defn}
\|f\|_{H^s}^2=\|f\|_{L^2}^2+ \iint \frac{\abs{f(x+t)-f(x)}^2}{\abs{t}^{d+2s}}\,dt\,dx.
\end{equation}
Note that if we are considering $f=1_A$, the indicator function of a set $A$, then the exponent $2$ does not matter because the numerator takes only 
values zero and one.
Thus 
$$ 
\|1_A\|_{H^s}^2=\mathrm{vol}(A)+ \iint \frac{\abs{1_A(x+t)-1_A(x)}}{\abs{t}^{d+2s}}\,dt\,dx.
$$
A similar definition works on closed Riemannian manifolds.
In particular, we will use the following norm which according to
\cite[Equation (2) and Lemma 2.21]{caselli2024fractional} is equivalent to the standard
Sobolev norm on $M$,
\begin{equation}
\label{SobolevManifold}
\|u\|_{H^s}^2=\|u\|_{L^2}^2+\iint_{M\times M} \frac{|u(x)-u(y)|^2}{\dist(x,y)^{d+ 2s}} (\vol \otimes \vol)(dx dy) ,   \end{equation}
where $\dist$ is the Riemannian distance on $M.$

We will need to restrict a set $A$ with $1_{A} \in H^s$ to an exponential chart.
Restricting $1_A$  to a ball in a chart cannot increase the $H^s$ norm  by  more than a bounded amount, provided $s < 1/2$. This is essentially because for such $s$ the family of functions $1_{B_{\delta}(0)}$ for $\delta<1$ has a uniform bound on its $H^s$ norm.  
\begin{proposition}\label{prop:H_s_cutoff}
Let $M$ be a connected, closed smooth Riemannian manifold of dimension $d$. Let $s \in [0, 1/2)$.
\begin{enumerate}
\item\label{item:Hs_loc_bound}
There exist $\delta_0, C_1 > 0$ such that   if 
$A\subseteq M$ is a set with $1_A\in H^s$,
then for any 
$z\in M$, any $0 < \delta<  \delta_0$, we have
\[
\|1_A1_{B_{\delta}(z)}\|_{H^s} \le \|1_A\|_{H^s}+C_1. 
\]

\item\label{item:Rd_hs_bound}
The same bound holds for  $A\subset \R^d$ and for $B^d(0, \delta)$ in place of $B_{\delta}(x)$. 

\item\label{item:Hs_chart_bound} 
Let $\phi \colon U \to V$ be a diffeomorphism between an open domain in $\R^d$ and an open set of $M$.  Then there exists a constant $C_2 > 0$ depending only on $\| \phi \|_{C^1}, \| \phi^{-1} \|_{C^1}$, such that for any $A\subset M$ with $1_A\in H^s(M)$, for any open set $B_{\delta}(x) \subset V$, we have 
$$ 
\| 1_{B_\delta(x) \cap A} \circ \phi \|_{H^s(\R^d)} \le C_2 \|1_A\|_{H^s(M)}+C_2.
$$ 
\end{enumerate}
\end{proposition}
\begin{proof} 
We will start with proving \eqref{item:Rd_hs_bound}, \eqref{item:Hs_loc_bound} is similar assuming that $\delta_0<\delta_{inj}$.
Denoting $\cB=B_\delta(z)$ for brevity
we have
\begin{align*}
\abs{1_A(y) 1_{\cB}(y)-1_A(x) 1_\cB(x)}&=
\abs{(1_A(y) 1_{\cB}(y)-1_A(y) 1_{\cB}(x))+(1_A(y) 1_\cB(x)-
1_A(x) 1_\cB(x))}
\\
&\leq \abs{1_\cB(y)-1_\cB(x)}+\abs{1_A(y)-1_A(x)}.
\end{align*}
Hence
\[\|1_A1_\cB \|_{H^s}^2
\leq 2+ \iint \frac{\abs{1_\cB(y)-1_\cB(x)}}{|x - y|^{d+2s}} dxdy +
\iint \frac{\abs{1_A(y)-1_A(x)}}{|x-y|^{d+2s}} dx dy.
\]
Since $1_\cB\in H^s$ for $s<1/2$, 
the first term is bounded proving \eqref{item:Rd_hs_bound}. 

Part \eqref{item:Hs_chart_bound} follows from parts \eqref{item:Hs_loc_bound} and \eqref{item:Rd_hs_bound} and the fact that $H^s$ is preserved by Lipschitz maps 
(which is apparent from \eqref{SobolevManifold}).
\end{proof}

The proof of the main theorem will crucially use the following bound, which gives a quantitative estimate on the density points of a set.
\begin{lemma}\label{lem:H_s_quantitative_density}
There exists $C_0>0$ depending only on $d$ such that for any $s \in (0, 1)$,
if $A$ is a bounded Borel subset of $\R^d$, of volume at most $1$, and $1_{A}\in H^s$,
then 
\begin{equation}
\label{ConvAppr}
\|1_A-1_A*\psi_{\delta}\|_{L^1}\le C_0 \delta^{2s}\|1_A\|^2_{H^s},
\end{equation}
where $\psi_{\delta}$ is the normalized indicator function of $B^d(0,\delta)$, i.e.\ $\psi_{\delta}=1_{B^d(0,\delta)}/\vol(B^d(0,\delta))$. 
\end{lemma}

\begin{proof}
Let $c_d = \vol(  B^d(0,1))$. We get
\begin{align*}
	&\|1_A-1_A*\psi_{\delta}\|_{L^1}\le \int_{\R^d} \int_{B^d(0,\delta)}
\frac{\abs{1_A(x)-1_A(x+t)}}{c_d \delta^d} dt dx 
\end{align*}
\hskip23mm $\DS 
\leq c_d^{-1} \iint \frac{\abs{1_A(x)-1_A(x+t)}}{|t|^{d+2s}}
\delta^{2s} dxdt\leq c_d^{-1} \delta^{2s} \|1_A\|_{H^s}^2.
$
\end{proof}

 Given a set $A$ we call a point $x\in A$ {\em  a point of density $c_0$ at scale $\delta$ of $A$} if $$\vol(A\cap B_{\delta}(x))/\vol(B_{\delta}(x)) >  c_0.$$ 
 We denote the set of $x$ satisfying the above inequality by $A_{\delta, c_0}.$

 \begin{corollary}
 \label{CrHD} 
 Let $s \in (0, 1)$. There exists $C>0$ such that if $A \subset \R^d$ is a positive volume subset with $1_A\in H^s(\R^d)$ then 
$\DS \vol(A\setminus A_{\delta, 1-\delta^s})\leq C\delta^s \|1_A\|_{H^s}^2.$ 
In particular,
$A_{\delta, 1-\delta^s}\neq \emptyset$ if $\delta$ is sufficiently small.
The same statement holds for $A \subset M$ with $1_A \in H^s(M)$.
 \end{corollary}

\begin{proof}
First let us assume $A \subset \R^d$. If $x\in A\setminus A_{\delta, 1-\delta^s}$ then 
$\DS 1_A-1_A*\psi_{\delta}\geq \delta^s$. Now the result follows from \eqref{ConvAppr} and Markov's inequality.

Now assume we have $A \subset M$ with $1_A \in H^s(M)$. 
Choose a partition of unity on $M$ subordinate to a finite covering by charts in $\R^d$. We can deduce the required estimate by applying the above statement for $\R^d$ to each chart. While the image of $B_{\delta}(x)$ in a chart may not correspond to round ball, it can be covered by a bounded number of round balls of radius comparable to $\delta$ due to the finiteness of charts. 
\end{proof}

  The following lemma says that if $1_A$ is in $H^s(\R^d)$ and $\wt{A}$ fills a significant portion of $A$, then  $\wt{A}$ has many points of relatively high density at small scales.

    \begin{lemma}\label{lem:contained_H_s_density_points}
    Fix $C>0$, $s \in (0, 1)$, $c_1>0$. Then there exists $c_0>0$ such that
    if $A\subset   \R^d$ is a   bounded Borel subset with $\|1_A\|_{H^s} \leq C$  and volume at most $1$, then the following holds.
     Suppose that $\wt{A}$ is a Borel subset of $A$ with $\vol(\wt{A})>c_1\vol(A)$.
     Then
    $ 
    \vol(\wt{A}_{\delta^{2d/s},c_0})>  \vol(\wt{A}) / 2 -   c_0^{-1} \delta^{2d}. 
    $
    \end{lemma}
 
\begin{proof}
\hspace{-.15cm}	Denote $\rho\!=\!\delta^{2d/s}$.
Let $A^{<1/2}$ be the set of points $x\in A$ satisfying 
$\DS \vol(B_{\rho}(x)\cap A)<\frac{\vol(B_{\rho}(x))}{2}$. 
Note that if $x \in A^{<1/2}$, then $ (1_A*\psi_{\rho})(x) < 1/2$, where, as before, 
$\DS \psi_{\rho}=1_{B_{\rho}(0)}/\vol(B_{\rho}(0))$. 
Hence Markov's inequality and Lemma \ref{lem:H_s_quantitative_density} show that 
$$
\vol(A^{<1/2})<2\delta^{2d}C^2  C_0,
$$
where $C_0$ is given by Lemma \ref{lem:H_s_quantitative_density}.

Now let $\wt{A}^{< 2c_0}$ be the set of points $x$ in $\wt{A}\setminus A^{<1/2}$ such that 
\[
\vol(B_{\rho}(x)\cap \wt{A})< 2c_0\vol(B_{\rho}(x)\cap A) . 
\]
By definition, we have 
\begin{equation}\label{eqn:atilde}
\wt{A}\setminus \left[A^{<1/2}\cup \wt{A}^{< 2c_0}\right] \subset \wt{A}_{\rho, c_0}.
\end{equation}  

By the Vitali Covering Lemma, there is a finite covering of $\wt{A}^{<  2c_0}$ with  multiplicity less than $C(d)$  by $\rho$-balls (indexed by $I$) centered at points of $\wt{A}^{< 2c_0}$.
  We can estimate the volume of $\wt{A}^{< 2c_0}$ by using this cover:
\begin{align*}
\vol(\wt{A}^{< 2c_0})&\le \vol\left(\bigcup_{x\in I} B_{\rho}(x)\cap \wt{A}\right) 
< \sum_{x\in I}  2 c_0\vol(B_{\rho}(x)\cap A)  
\le  2 C(d) c_0\vol(A) 
\le  2C(d)c_0c_1^{-1}\vol(\wt{A}) .
\end{align*}
 Thus the set in \eqref{eqn:atilde} has measure at least
$\DS
(1-2 C(d)c_0c_1^{-1})\vol(\wt{A})-2\delta^{2d}C^2 C_0.
$
We conclude the proof by letting $c_0$ be sufficiently small depending on $d$ and $c_1$.
\end{proof}

 The relevance of the results presented above comes from the following proposition.   Recall that a set $A$ is a totally ergodic component of volume if for every power $\mu^{ q}$ of $\mu$, the random dynamics of $\mu^{ q}$ on $\vol\vert_A$ is ergodic. 

\begin{proposition}\label{prop:ergodic_decomposition} 
Suppose that $\mu$ is coexpanding on average (see Definition \ref{DefCoA}). 
Then for some $s>0$, there exist positive integers $l,q$, and disjoint Borel subsets $A_1,\ldots,A_{l}$ such that 
\begin{enumerate}[leftmargin=*]
    \item 
$\vol(A_1\cup\cdots\cup A_l)=1$,
    \item  
    $1_{A_i}\in H^s(M)$ for each $i$, 
    \item 
    Each $A_i$ is   $\mu^q$-almost surely invariant and the restriction of the random dynamical system driven by $\mu^q$ to $A_i$
    is totally ergodic.
\end{enumerate}
\end{proposition}

\begin{proof}
  The decomposition with required properties is established in
\cite[Corollary 7.5]{dewitt2025conservative} except it is not explicitly stated there that $1_{A_i}$ belongs to $H^s(M)$.
To see this fact, take a function $\phi\!\in\! H^s(M)$ such that 
$\DS \int_{A_i} \!\phi(x) dx\!=\!1$ and 
$\DS \int_{A_j} \! \phi(x) dx\!=\!0$ for $j\neq i$. Let $\cL$ be the generator of the dynamics driven by $\mu^q$, that is
$\DS (\cL \phi)(x)=\int \phi(fx) d\mu^q(f). $
Then 
$\DS \frac{1}{N} \sum_{n=0}^{N-1} \cL^n(\phi)$ converges to $\vol(A_i)^{-1}1_{A_i}$ in $L^2$ due to the ergodic theorem, and it converges to 
some function $\bar\phi$ in $H^s(M)$ due to the spectral decomposition of $\cL$ established in \cite{dewitt2025conservative}.
It follows that $\bar\phi= \vol(A_i)^{-1} 1_{A_i}$ almost surely, and so $1_{A_i}\in H^s(M)$.
\end{proof}

\begin{remark}
The requirement for $C^{\infty}$ maps in this paper stems from the assumptions made by the authors of \cite{dewitt2025conservative} that are used in the proof of  Proposition \ref{prop:ergodic_decomposition}. The argument for proving Main Theorem \ref{thm:main_thm} holds for $C^2$ maps, provided that the conclusions of Proposition \ref{prop:ergodic_decomposition} hold under this lower regularity.
\end{remark}

\begin{remark}\label{rem:density_points_exponential_chart}

To transfer between estimates on the manifold and estimates on Euclidean space, we will often use an exponential chart $\exp_x \colon B^d(0, \delta_{inj}) \to M$ to identify a neighborhood of $x$ in $M$ with a neighborhood of the origin in $\R^d$.

  We will make some harmless abuses of notation relative to an exponential chart $\exp_x$.
    For example,    for a subset $A\subseteq M$, we will also write $A$  for $(\exp_x)^{-1}(A)\cap B^d(0, \delta_{inj})$.
    If $1_A \in H^s(M)$ with $s \in (0, 1/2)$ and $B_{\delta}(y) \subset \exp_x( B^d(0, \delta_{inj}) )$, then by Proposition \ref{prop:H_s_cutoff} the set $A \cap B_{\delta}(y)$, viewed as a subset of $\R^d$ through $\exp_x$, has indicator function with $H^s$-norm bounded by $C(M) (\| 1_{A}\|_{H^s(M)} + 1)$.
   
        We will  use implicitly that all the exponential charts on $M$ have bounded $C^2$-norms, and in particular, have uniformly bounded Jacobians.  For instance, if a curve $\gamma \subset M$ is  $\fC$-good, then $\exp_x^{-1}(\gamma)$ has $C^2$-norm bounded by $C(M) \fC$. If $B_{\delta}(y) \subset \exp_x(B^d(0, \delta_{inj}))$, then 
        the density $\vol(A \cap B_{\delta}(y)) /  \vol( B_{\delta}(y) )$ is comparable to $\vol(\exp_x^{-1}(A \cap B_{\delta}(y))) /  \vol( \exp_x^{-1}(B_{\delta}(y)) )$ up to factor depending only on $M$. For this reason, we may abusively use $\vol$ to denote the volume measure on $\R^d$ as well.
        Finally,  we will use the Euclidean metric to measure the angle between two vectors at different points in $\exp_x(B^d(0, \delta_{inj}))$. This is consistent with the notation in 
        Definition~\ref{def goodandaligned}. 
\end{remark}

\begin{proposition} \label{lem:configuration_of_local_laminations}
		Suppose that $\mu \in \cG$. 
         Let $A$ be a $\mu$-almost surely invariant subset with positive volume. 
		There exist $s \in (0, 1/2)$, $\fD>0$,  $\theta_0>0$ such that  for any $\eps >0$ there exist $\delta_1 \in (0, \delta_{inj})$, $\eta_1 >0$ such that for  any $x\in M$ 
        and any $0<\delta<\delta_1$ such that $\vol(A \cap B_{\delta}(x))\!>\! 0$,  there exist 
          ${\bf v} \!=\!  (v_1,\ldots,v_d )\! \in\! ( T^1_xM )^d$ 
        with $\abs{\det(v_1,\ldots,v_d)}> \theta_0$, and a subset 
        $\Omega_1 \subset \Omega^{d}$ with $\PP^{\otimes d}(\Omega_1) > \eta_1$ such that 
        for every $(\omega^1, \ldots, \omega^d) \in \Omega_1$, we have the following:
	\begin{enumerate}[leftmargin=*]
		\item 
		for every $1 \leq i \leq d$, there is a 
        $(\omega^i, \fD)$-stably 
        laminated Borel subset $K_i\subset B_\delta(x)$ that is  
     $\eps$-aligned with $v_i$ 
     such that $\DS \bigcup_{i = 1}^{d} K_i \subset A$ modulo a volume null set, and
     $\DS K := B_{\delta}(x) \cap A \cap \bigcap_{i = 1}^{d} K_i $  
     satisfies $\vol(K) > \eta_1 \vol(B_{\delta}(x) \cap A)$;
		\item let $\widetilde K $ be the set of points $y \in K$ such that 
		\aryst
		\vol(B_{\delta^{2d/s}}(y) \cap K) > \eta_1 \vol(B_{\delta^{2d/s}}(y)),
		\earyst
		Then $\vol(  \widetilde K ) > \eta_1 \vol(K) -  \eta_1^{-1} \delta^{2d}$.
		In particular, we have 
		\aryst
			\vol(  \widetilde K ) > \eta_1^2 \vol(B_{\delta}(x) \cap A) -  \eta_1^{-1} \delta^{2d}.
		\earyst
		\end{enumerate}
\end{proposition}

\begin{proof}
	Let $s \in (0, 1/2)$ be a constant satisfying the conclusion of Proposition~\ref{prop:ergodic_decomposition}.
    
Fix an arbitrary    $x \in M$. 
    Let   $\fD > 0$, $ \theta_0 > 0$ be given by Proposition~\ref{prop:laminations_exist}. For every $\eps > 0$, we let  $\delta_0 > 0$, $\eta_0 > 0$ be given by Proposition \ref{prop:laminations_exist} and take some $\delta \in (0, \delta_0)$.
	Applying Proposition \ref{prop:laminations_exist} to $ B_{\delta}(x)  \cap A$ in place of $A$,
 we obtain some ${\bf v} = (v_1,\ldots,v_d )\in ( T^1_xM )^d$ with $\abs{\det(v_1,\ldots,v_d)}>\theta_0$
 and $\Omega_1  \subset \Omega^{d} $
 such that for each $(\omega^1, \ldots, \omega^d) \in \Omega_1$ and each $1 \leq i \leq d$, 
 there is a 
 $(\omega^i, \fD)$-stably laminated
 Borel subset $K_i \subset B_{\delta}(x)$ 
 $\eps$-aligned with $v_i$, such that 
 $$
 \vol \left(B_{\delta}(x) \cap A \cap \bigcap_{i = 1}^{d} K_i \right) > \eta_0 \vol(B_{\delta}(x) \cap  A).
 $$ 
	By Theorem \ref{thm ergodiccomponent},  after possibly modifying $\Omega_1$ by a conull subset, for every 
    $(\omega^1, \ldots, \omega^d) \!\!\in\!\! \Omega_1$, and every $1 \leq i \leq d$, $K_i$ is contained in $A$ modulo a volume null set.
Then, letting $\eta_1 < \eta_0$,
	Item (1) follows.
	
 	  Now, we have $1_A\in H^s$ by Proposition \ref{prop:ergodic_decomposition}.
      In the exponential chart at $x$, we may view $B_{\delta}(x)$, $A$ as subsets on $\R^d$. Then
      $\| 1_{B_{\delta}(x) \cap A} \|_{H^s(\R^d)}$ is uniformly bounded in $\delta$     due to Proposition \ref{prop:H_s_cutoff} (see  Remark \ref{rem:density_points_exponential_chart}).
    We apply Lemma \ref{lem:contained_H_s_density_points} 
    for $B_{\delta}(x) \cap A$ in place of $A$, for $K$ in place of $\widetilde{A}$ and for $c_1 = \eta_0$.
    Then by letting $\eta_1 < \min(c_0, \eta_0, 1/2)$ ($c_0$ is produced by Lemma \ref{lem:contained_H_s_density_points}), we obtain Item (2). This finishes the proof.  
\end{proof}

\section{A multilinear Kakeya Estimate}\label{sec:geometric_lemma_following_guth}

In this section, we will introduce the  multilinear Kakeya estimate from harmonic analysis, 
Theorem \ref{thm:tao_polynomial_cylinders}, which is crucial in our proof. This theorem is an immediate corollary of Theorem~\ref{thm:taos_theorem}, which is a restatement of a theorem of Tao in \cite{tao2020sharp}. Roughly speaking, this theorem estimates the measure of an intersection of $d$ transverse families of nearly parallel tubes 
in terms of the measures of the cross sections of these families.

We will now  explain precisely the statement of \cite[Thm.~1.3]{tao2020sharp}.
First, we describe a measure on a family of approximate tubes. 
Suppose that $\phi \colon B^d(0,\bA)\to \R^{d-1}$ is a $C^2$ submersion. Then assuming the derivatives of $\phi$ are \lq\lq suitably controlled\rq\rq, which we will explain in a moment, it follows that $\phi$ defines an approximate tube $\phi^{-1}(B^{d-1}(0,t))\subset \R^d$ of radius proportional to $t$. Note that the indicator function of this tube is $1_{B^{d-1}(0,t)}\circ \phi$.  By using an indexing measure space $(\Omega,\mu)$, we may consider a measurable family of functions $\Big ( x \mapsto \phi[x](\omega) \Big )$ indexed by   $\omega\in \Omega$. Naturally, we can integrate over this space to define the (weighted) indicator function of the family of approximate tubes:
\[
\int_{\Omega} 1_{B^{d-1}(0,t)}(\phi[x](\omega))\,d\mu(\omega).
\]
Note that this expression is a function of $x$.

We now make more precise the notion of {\em suitably controlled} in the above paragraph. 
Note that $\phi$ is a submersion when the matrix  $D_x\phi$ has rank $d-1$ at every point. This matrix also controls the geometry of the approximate tubes $1_{B^{d-1}(0,t)}\circ \phi$. 
By controlling the non-zero singular values of this matrix, one ensures that the preimage of $B^{d-1}(0,t)$ will be 
  an approximate tube with transverse size comparable to $t$.
We will also   need to control the second derivatives of $\phi$.

In the statement of the theorem, we will have $d$ different collections of tubes, so that the tubes in the different collections are transverse to each other.
Note that the vectors in the kernel of $D\phi$ are 
 ``tangent  to the axis of the tube." Thus if we have $d$ submersions $\phi_i\colon \R^d\to \R^{d-1}$, $1\le i\le d$, and a point $x\in \R^d$ we can ask that these tubes be mutually transverse in the following sense. Suppose that $n_1,\ldots,n_d$ are unit vectors such that $n_i\in \ker D_x\phi_i$. Then the natural way to quantify how mutually transverse these vectors are is the determinant $\abs{\det(n_1,\ldots,n_d)}$.

We now state Tao's main result from \cite{tao2020sharp}   using our notation. 

\begin{theorem}\label{thm:taos_theorem}
\cite[Thm.~1.3]{tao2020sharp} Let $\frac{1}{d-1}<\bp\le \infty$ be an exponent, and let $\bA\ge 2$ be a parameter. Let $\cV\subset B^d(0,\bA)$ be an open subset of the ball $B^d(0,\bA)$ of radius $\bA$ centered at the origin in $\R^d$. 
Let $(\vec{\Omega},\vec{\mu})$ be a $d$-tuple of finite measure spaces $(\Omega_j,\mu_j)$, and for each $j\in \{1,\ldots,d\}$, let $\phi_j\colon \cV\times \Omega_j\to \R^{d-1}$ be a measurable map (denoted 
$(x,\omega_j)\mapsto \phi_j[x](\omega_j)$) obeying the following axioms:
\begin{enumerate}[leftmargin=*]
    \item (Regularity) For each $j\in \{1,\ldots,d\}$ and $\omega_j\in \Omega_j$, the map $\phi_j[\cdot](\omega_j)\colon x\mapsto \phi_j[x](\omega_j)$ is a $C^2$ map   whose $C^2$ norm is bounded by $\bA$.
    \item (Submersion) For each $j\in \{1, \ldots, d\}$, $\omega_j\in \Omega_j$, and $x\in \cV$, the derivative matrix $D_x\phi_j[x](\omega_j)\in \R^{(d-1) \times d}$ is of full rank, with the $(d-1)$ non-trivial singular values lying between $\bA^{-1}$ and $\bA$.
    In particular, $\phi_j[\cdot](\omega_j)\colon \cV\to \R^{d-1}$ is a $C^2$-submersion;
    \item (Transversality) For any $x\in \cV$ and $\omega_j\in \Omega_j$ for $j\in \{1,\ldots,d\}$,
    \begin{equation*}  \abs{\det(n_1,\ldots,n_d)}
  \ge \bA^{-1}
    \end{equation*}
    whenever $n_j\in \R^d$ is a unit vector in the nullspace of $D_x\phi_j[x](\omega_j)$.
\end{enumerate}
    
     Define
    $\DS
    \cV_{1/\bA}\coloneqq \{x\in \cV |  B^d(x,1/\bA)\subset \cV\}.
    $
    Then for any $0<t\le \bA$, one has
    $$
    \left\|\prod_{1\le j\le d} \int_{\Omega_j} 1_{B^{d-1}(0,t)}(\phi_j[x](\omega_j))\,d\mu_j(\omega_j)\right\|_{L^\bp(\cV_{1/\bA})}\!\!\!\!\!\!\!\le \bA^{C(d)} 
    \left(d-1-\frac{1}{\bp}\right)^{-C(d)}t^{\frac{d}{\bp}} \prod_{1\le j\le d} \!\!\! \mu_j(\Omega_j). 
    $$
    Here $C(d)$ denotes a constant depending only on $d$.
\end{theorem}

When we apply the above theorem, we will consider a fixed $\bA$ that is large enough that the curves in distinct families we consider will not intersect near the boundary of the ball $B^{d}(0,\bA)$. 

 In the above formulation, a curve at the center of the tube is described as a fiber of a mapping from
$\mathbb{R}^d$ to $\mathbb{R}^{d-1}$. For our purposes it is more convenient to work with curves in
a parametric form. Therefore we now restate the  multilinear Kakeya estimate adapted to the present setting.

Given $\delta > 0$, we denote by ${\bf X}(d, \delta)$
the space of arc-length parametrized $C^2$ embeddings of $[- \delta, \delta]$ into $\R^d$, equipped with the  $C^2$ topology. 

\begin{theorem}\label{thm:tao_polynomial_cylinders}
Fix $d\in \N$.
For every 
$p>1$, $B_1, \theta >0$ there exist $\delta_0,  \epsilon_0, C_1>0$ such that for every 
$0 < \delta<\delta_0$, $0 < \rho < \delta / 2$,
the following is true.
	Let $( v_1, \ldots, v_d )$ be a frame of unit vectors in $\R^d$ such that $\abs{ \det(v_1, \ldots, v_d)} > \theta$.
	For each $j \in \{1, \ldots, d\}$, let $\mu_j \in \mathcal{M}({\bf X})$ be  a finite Borel measure on  
    ${\bf X} =  {\bf X}(d,    4 \delta)$  
    satisfying that for each $\ell\in \supp \mu_j$, we have $\|\ell\|_{C^{2}}\le B_1$ and 
  $\DS \sup_{t \in (- \delta , \delta)} \angle(v_j, \ell'(t))  <\epsilon_0$,
and  
$\im(\ell) \cap B^{d}(0, \delta) \subset \im(\ell[-2\delta, 2\delta])$.
Then 
$$\int_{B^{d}(0, \delta)} \left[\prod_{j=1}^d \int_{{\bf X}} \chi_{\ell, \rho } \,  \,d\mu_j(\ell)\right]^{\frac{p}{d-1}}\,d\vol 
 \le C_1  \rho^{d} \Big ( \prod_{j=1}^d 
\mu_j({\bf X}) \Big )^{\frac{p}{d-1}},
$$
where $ \chi_{\ell, \rho}$  is the indicator of the $\rho$ tubular neighborhood of $\ell$.
\end{theorem}

\begin{proof}
	Let $\bA > 2$ be a large constant to be determined later depending only on $d, B_1, \theta$.

Denote $\widetilde{\bf X} \!=\! {\bf X}(d, 4\bA)$.
To each $\ell \in {\bf X}$, we associate  $\widetilde\ell \in \widetilde{\bf X}$ given by 
$\widetilde{\ell}(t) \!\!= \! \! \bA \delta^{-1} \ell( \delta t / \bA)$. It is direct to see that for every $j$ and every $\ell \in \supp(\mu_j)$, \; $\widetilde\ell$ is also parametrized by length, and 
$$
\| D^2 \widetilde\ell\|  \le  B_1 \delta /\bA, \quad
\sup_{t \in (- 4\bA, 4\bA)} \!\!\! \angle(v_j, (\widetilde{\ell})'(t)) <   \epsilon_0.
$$
Let $\widetilde\rho :=  \bA  \rho / \delta$. Then 
$\chi_{\widetilde\ell, \widetilde\rho}(\cdot) =   \chi_{\ell, \rho}( \bA^{-1} \delta \cdot)$ is the indicator function of the $\widetilde\rho$ tubular neighborhood of $\widetilde\ell$.

Fix some $j \in \{1, \ldots, d\}$. 
Denote $V_j  = \R v_j$. 
Let $\ell \in \supp \mu_j$. 
By our hypothesis, the intersection between $B^d(0, \bA)$ and the image of $\widetilde\ell$, denoted by $\gamma_{\ell}$,  is contained in the image of $\widetilde\ell|_{[-2 \bA, 2 \bA]}$.  Due to the $\epsilon_0$-alignment hypothesis and bound on the $C^2$ norm of $\widetilde\ell$, we can choose a $C^2$ function  
$g_{\ell}\colon V_j(3\bA) \to V_j^{\perp}$ such that $\| g_{\ell} \|_{C^2} \le B_1$ and   $\gamma_{\ell}$  is contained in the graph of $g_{\ell}$   (Here $V_j(3\bA)$ denotes the ball in $V_j$ centered at the origin with radius $3\bA$).  Our choice of $g_{\ell}$ can be made to depend measurably on $\ell$  by defining it using the Taylor polynomial at the endpoint. Due to the bound $\|D^2\wt{\ell}\|\le B_1 \delta/\bA$, as long as $\delta_0$ is sufficiently small, the graph of $g_{\ell}$ is also $2\epsilon_0$-aligned.
We define $\widetilde\phi_{\ell}\colon B^d(0, 3\bA) \to \R^{d-1}$ by
$\widetilde\phi_{\ell}(y, z)=z-g_{\ell}(y)$ for every $x=(y,z)\in V_j(3\bA) \times V_j^\perp$.  By construction, $\widetilde\phi_{\ell}$ depends measurably on $\ell$.
 An easy computation shows that
  the nullspace of 
$D_x\widetilde{\phi}_{\ell}$ at every point is parallel to the tangent space of $\widetilde\ell$ at some point, and hence is $ 2\epsilon_0$-close to $v_j$.
Thus the assumptions of Theorem \ref{thm:taos_theorem} are satisfied when $\epsilon_0$ is sufficiently small depending on $d, \theta$, and $\bA$ is sufficiently large depending on $d, B_1$ and $\theta$.
 Moreover, restricted to $B^d(0, \bA)$, we have
\begin{equation}\label{eqn:chi_lower_bound2rho}
\chi_{\widetilde\ell,  \widetilde \rho}\le  
 1_{B^{ d-1}(0,    2 \widetilde\rho)}\circ \widetilde\phi_{\ell}.
\end{equation}
 Applying   Theorem \ref{thm:taos_theorem} 
with  $  \mathcal{V}=B^d(0, \bA+\frac{1}{\bA})$,  $t= 2 \widetilde\rho$ and $\bp=p/(d-1)$ yields 
\begin{equation}\label{eqn:application_of_tao1}
\int_{B^d(0,  \bA)} \left[\prod_{1\le j\le d} \int_{{\bf X}} 1_{B^{d-1}(0, 2 \widetilde\rho)}(\widetilde\phi_{\ell}(x))\,d\mu_j(\ell)\right]^{\frac{p}{d-1}} \!\!\!\! d\vol(x) \le 
 C \widetilde\rho^d \prod_{1\le j\le d} \mu_j({\bf X})^{\frac{p}{d-1}}. 
\end{equation}
By \eqref{eqn:chi_lower_bound2rho} and by a change of variable $x = \bA y/\delta$, the left hand side above is bounded from below by
$$
\bA^d   \delta^{-d}	\int_{B^d(0, \delta)} \left[\prod_{1\le j\le d} \int_{{\bf X}}   \chi_{ \ell,   \rho}(y) \,d\mu_j(\ell)\right]^{\frac{p}{d-1}} \!\!\!\! d\vol(y).
$$
The right hand side of \eqref{eqn:application_of_tao1} is
$\DS 
 C( \bA  \rho / \delta)^d  \Big ( \prod_{1\le j\le d} \mu_j({\bf X}) \Big )^{\frac{p}{d-1}}.  $
This finishes the proof.
\end{proof}

\section{Non-ergodicity implies the density gap}\label{sec:main_estimate}

  Here we prove our main result, Theorem \ref{thm:main_thm}.

\begin{proposition}\label{prop:quantitative_density_est}
Suppose that $\mu \in \cG $ and that $A$ is a $\mu$-almost surely invariant set of  positive volume. 
Given $p>1$
there exist $\delta_0,C_p>0$ such that 
for any $0<\delta<\delta_0$ and any   $x\in M$, 
\begin{equation}
\frac{\vol(A\cap B_{\delta}(x))}{\vol(B_{\delta}(x))} -C_p \delta^{  d} 
\le C_p  \delta^{-\eps_p} \left(\frac{\vol(A\cap B_{\delta}(x))}{\vol(B_{\delta}(x))}\right)^{\frac{pd}{d-1}}
\end{equation}
where $\DS \eps_p=(p-1) d \left(\frac{2d}{s}-1\right)$,
and $s$ is from Proposition \ref{lem:configuration_of_local_laminations}.
\end{proposition}

\begin{remark}
    The precise value of $\eps_p$ is not important for our argument; we just use that 
    $\DS \lim_{p\downarrow 1} \eps_p=0$.
\end{remark}

\begin{proof}
  With the convention in Remark \ref{rem:density_points_exponential_chart},
  we may identify the subset $B_{\delta}(x)$ of $M$ as a subset of $\R^d$, using the exponential chart at $x$. We may also abusively use $\vol$ to denote the volume on $\R^d$, as this will only change the ratio $\frac{\vol(A\cap B_{\delta}(x))}{\vol(B_{\delta}(x))}$
  by a factor depending only on $M$.

We will let $C$ and $c$ be generic positive constants, depending only on $\mu, M, d$   and $p$, which may vary from line to line.

We may assume without loss of generality that $\vol(A \cap B_{\delta}(x) ) \!>\! 0$.
Let $s > 0$, $\fD > 1$ and $\theta_0 > 0$ be given by Proposition~\ref{lem:configuration_of_local_laminations}, applied to $A$.
Let $\eps > 0$ be a small constant so that for every 
${\bf v} = (v_1,\ldots,v_d )\in ( \R^d)^d$ 
with $|\det(v_1, \cdots, v_d)| > \theta_0$, for every ${\bf v}' = (v'_1,\ldots,v'_d )\in ( \R^d )^d$  
with $\DS \max_{1 \leq i \leq d} | \angle(v_i, v'_i) | < \eps$, we have $|\det(v'_1, \cdots, v'_d)| > \theta_0/2$. As there are only finitely many choices for $A$ (see Proposition \ref{prop:ergodic_decomposition}), we may consider $s, \fD, \theta_0, \eps$ as parameters depending only on $\mu$. 

Let $\delta > 0$ be a small parameter and $\rho  = \delta^{2d/s}$.

Applying Proposition~\ref{lem:configuration_of_local_laminations} to $A$ and $\eps$, we obtain $\eta_1 > 0$ (depending only on $\mu$) and 
a set of $(\omega^1, \cdots, \omega^d)$ with positive $\PP^{\otimes d}$-measure such that for almost every such $(\omega^1, \ldots, \omega^d)$, there exist   
${\bf v} = (v_1,\ldots,v_d )\in ( \R^d )^d$,
$(K_i \subset B_{\delta}(x))_{i = 1}^{d}$ such that  
\enmt
\item for every $i$, $K_i$ is $(\omega^i, \fD)$-stably laminated and $\eps$-aligned with $v_i$, and 
$\DS \bigcup_{i = 1}^{d} K_i$ is contained in $A$ modulo a volume null set,
\item
$\DS \widetilde{K} \subset K \coloneqq B_{\delta}(x) \cap A \cap \bigcap_{i = 1}^{d} K_i$ satisfies
\begin{equation}\label{eqn:Ktilde_full}
\vol(  \widetilde K ) > \eta_1^2 \vol(B_{\delta}(x) \cap A) -  \eta_1^{-1} \delta^{2d},
\end{equation}
and for every $y \in \widetilde K$, we have 
\begin{equation}\label{eqn:Ki_full}
\vol(B_{\rho}(y) \cap K_i) > \eta_1 \vol(B_{\rho}(y)), \quad 1 \leq i \leq d.
\end{equation}
\eenmt

 Fix an arbitrary $1 \leq i \leq d$. 
 Since the $K_i$ are $(\omega^i,\fD)$-stably laminated,
 \eqref{eq decomposition} and Proposition \ref{prop Conditional}
 allow us to decompose the restriction of the volume on $K_i$ as follows: for each Borel set $E\subset K_i$ 
 \begin{equation}
  \label{VolNuL}   
 	 		\vol(E) = \int_{\mathbf{X}} \nu_\ell(E)  d\mu_i(\ell)
  \end{equation}
where    $ \mathbf{X} = {\bf X}(d, 4\delta)$     and the $\mu_i$ are factor measures supported on $\omega^i$-strong stable manifolds satisfying the hypotheses of Theorem \ref{thm:tao_polynomial_cylinders}, and the $\nu_\ell$
are probability measures on $\ell$ satisfying
\begin{equation}
\label{DensityDelta}
  \frac{1}{C \delta} \leq     \frac{d\nu_\ell}{d m_\ell} \leq \frac{C}{\delta} \ \mbox{ on $\ell \cap B_{\delta}(x)$},
\end{equation}
where $m_\ell$ is the arclength measure on $\ell$. Note that taking $E=K_i$ in \eqref{VolNuL} gives
\begin{equation}
\label{MuiKi}
\mu_i(\mathbf{X})=\vol(K_i). 
\end{equation}
  
By Theorem \ref{thm:tao_polynomial_cylinders} and \eqref{MuiKi},
\begin{equation}\label{eqn:nonlinear_est0}
\int_{B_{\delta}(x)}\!\!\! \left[\prod_{i=1}^d (\int \chi_{\ell, \rho}(y) d\mu_i(\ell)) \right]^{\frac{p}{d-1}}\!\!\!\!d\vol (y)
\le   C 
\rho^{d}  \prod_{i=1}^d 
\left[    \vol(K_i) \right ]^{\frac{p}{d-1}}\!\!.
\end{equation}
We now relate the two sides of the above inequality to $\vol(A\cap B_{\delta}(x))$.  
\vskip2mm

\noindent\textbf{Right hand side of \eqref{eqn:nonlinear_est0}.}  
It is clear that 
$\DS
 \vol(K_j) \leq \vol(A \cap B_{\delta}(x)).
$
Hence the right hand side of \eqref{eqn:nonlinear_est0} is bounded from above by 
\ary \label{eq RHS}
     C
     \rho^{d} (  \vol(A \cap B_{\delta}(x)) )^{p d/(d-1)}.
\eary

\noindent\textbf{Left hand side of \eqref{eqn:nonlinear_est0}.} 
We bound the left hand side by considering the integrand restricted to $\widetilde K$. For every $y \in \widetilde{K}$, every $1 \leq i \leq d$, and every $\ell\in \supp(\mu_i)$, we have 
$$ 
\chi_{\ell, \rho}(y) \geq \frac{c\delta}{\rho} \nu_\ell(B_{\rho }(y)). 
$$
Integrating with respect to $\mu_i$ and using \eqref{DensityDelta} and the fact that $\widetilde{K}$ consists of density points 
of $K_i$ we conclude that
$$
\int \chi_{\ell, \rho}(y) d\mu_i(\ell)  \geq \frac{c\delta}{\rho} \vol(K_i\cap B_{\rho}(y)) \geq
 c  \delta \rho^{d-1} . 
 $$
Thus the left hand side of \eqref{eqn:nonlinear_est0} is bounded from below by
\ary \label{eq LHS}
c \vol(\widetilde{K}) \rho^{d p}  \delta^{pd/(d-1)}.
\eary

\noindent\textbf{Conclusion.}  
Comparing \eqref{eq RHS} and \eqref{eq LHS}, 
we get
\aryst
  \vol(\widetilde{K})    \delta^{pd/(d-1)} \rho^{(p-1) d} \leq  C
  (  \vol(A \cap B_{\delta}(x)) )^{pd/(d-1)}.
\earyst
Dividing both sides by $\vol(B_{\delta}(x))^{pd/(d-1)}=C\delta^{d^2p/(d-1)}$, and using \eqref{eqn:Ktilde_full} and \eqref{eqn:Ki_full}, we obtain:
\begin{align*}
\left[\eta_1^2 \frac{\vol(B_{\delta}(x) \cap A)}{\vol(B_{\delta}(x))} -  \frac{C \delta^{ d}}{\eta_1} \right]
\delta^{d(1-p)} \rho^{(p-1) d}
\leq    \frac{\vol(\widetilde{K})}{\vol( B_{\delta}(x) ) } \delta^{(1-p)d} \rho^{(p-1) d}
\leq C 
\left( \frac{ \vol(A \cap B_{\delta}(x)) }{ \vol(B_{\delta}(x))} \right)^{\frac{dp}{d-1}}\!\!\!\!.
\end{align*}
This concludes the proof.
\end{proof}

We now prove our main theorem.

\begin{proof}[Proof of Theorem \ref{thm:main_thm}.]
Suppose for contradiction that the random system is not ergodic. Then there is a $\mu$-almost invariant subset $A$ with $\vol(A) \in (0, 1)$. 

Fix $p>1$ so close to 1 such that $\eps_p$ given by Proposition \ref{prop:quantitative_density_est} satisfies 
$s/(d-1) >\eps_p$.
  
Let $\delta > 0$ be a small parameter to be determined later depending on $p$ and $\mu$.
By Corollary~\ref{CrHD} (applied to both $A$ and $A^c$), by letting $\delta$ be small, we may find some $x_0 \in A^c $ and $x_1 \in A$ such that 
$\DS \frac{\vol(A\cap B_{\delta}( x_0 ) )}{\vol(B_{\delta}( x_0 ))} \leq \delta^s$ and  
$\DS \frac{\vol(A\cap B_{\delta}( x_1 ) )}{\vol(B_{\delta}( x_1 ))}\geq 1- \delta^s$. 
Let $\gamma \colon [0, 1] \to M$ be a continuous map with $\gamma(0) = x_0$ and $\gamma(1) = x_1$. Denote $q_{\delta}(t) =  \vol(A\cap B_{\delta}( \gamma(t)) )  / \vol(B_{\delta}(\gamma(t))) $.  
By the Intermediate Value Theorem there exists $\bar t$ with $q_{\delta}(\bar t)=\delta^s.$
Applying Proposition~\ref{prop:quantitative_density_est}  with $x= \gamma(\bar t)$
we conclude that
\begin{equation}
\label{SmallerDensity}    
\delta^s-  C_p \delta^{d}  \le C \delta^{spd/(d-1)-\eps_p} 
\leq C_p \delta^{sd/(d-1)-\eps_p} 
= C_p \delta^{s + (s/(d-1)  - \eps_p)}.
\end{equation} 
We obtain a contradiction if $\delta$ is small enough. 
This completes the proof.
\end{proof}

\appendix

\section{Construction of the splitting}\label{app:construction_of_splitting}

\subsection{Proof of Lemma \ref{lem:ahlfors_Ess}}

The following lemma transfers the gap condition into a fractional moment bound for ratio of derivatives growth.

\begin{lemma} \label{lem MomentBound}
	If $\mu\in \cG$,
	then there exist  $\kappa_1 = \kappa_1(\mu) > 0$, $\sigma = \sigma(\mu) > 0$ and $C_1  = C_1(\mu) > 1$ such that for any $x \in M$,  any $W \in \Gr(T_x M, d - 1)$, and any integer $n \geq 1$, we have
	\aryst
	\int  \Big ( \sup_{v \in W^{\perp}, \| v \| = 1} \| P_{D_xg( W)^{\perp}}(  D_xg( v) ) \| /  { \con(D_xg|W)} \Big )^{\sigma} d\mu^{* n}(g) < C_1 e^{- n \sigma { \kappa_1} / 2}.
	\earyst
	
\end{lemma}
\begin{proof}
The proof consists of several steps.

{\bf Step 1.}
     Using that $e^{-t} \le 1 - t/2 \leq e^{-t/2}$ for  $t > 0$ small enough, we obtain that for sufficiently small $\sigma, \kappa_1 > 0$
\begin{equation}
\label{Gap1Step}
	\int  \Big ( \sup_{v \in W^{\perp}, \| v \| = 1} \| P_{D_xg( W)^{\perp}}(  D_xg( v) ) \| /  { \con(D_xg|W)} \Big )^{\sigma} d\mu^{* N}(g) \le e^{- N \sigma \kappa_1 / 2},
    \end{equation}
where $N$  comes from \eqref{CoDim1Gap}. 

{\bf Step 2.} 
Let $(\cF_k)_{k \in \N}$ be the filtration on $\Omega$ given by the coordinate projections.  
	Consider a Markov chain on the $(d-1)$-Grassmannian bundle of $M$, given by the transition probabilities
	\begin{equation}
    \label{AppMC}
	P((x, W), Y ) = \mu(  g \mid Dg(x, W) \in Y  ).
	\end{equation}
	We use $\PP_{(x, W)}$ to denote the law of the Markov process initiated from $(x, W)$ and use $\EV_{(x, W)}$ for
    the expectation with respect to $\PP_{(x, W)}$.
	
	For each $x \in M$ and integer $m \geq 0$, we define 
	\begin{equation}
    \label{OrbitNotation}
	X_m = f^{mN}_{\omega}(x), \quad W_m = D_x f^{mN}_{\omega}(W) \in \Gr(T_{X_m}M, d - 1).
	\end{equation}   

We will prove by induction that for each natural number $m$
\begin{equation}
\label{ExpGap}
	\mathbb{E}_{(x,W)} \left[\left( \sup_{v \in W^{\perp}, \| v \| = 1} \| P_{W_{m}^{\perp}}(  D_x f^{mN}_\omega( v) ) \| /  
    { \con(D_x f_\omega^{mN}|W)} \right)^{\sigma}\right]  < e^{- mN \sigma \kappa_1 / 2}.
\end{equation}    
The base of the induction was established in Step 1. To carry out the inductive step suppose that the result holds for
$m$. Then using that for matrices $A_1, A_2\in GL_d(\R)$\;\; 
$\|A_1 A_2\|\leq \|A_1\| \|A_2\|,$ $\con(A_1 A_2)\geq \con(A_1) \con(A_2)$ 
and denoting by $\gamma_N(x, W)$ the left hand side of \eqref{Gap1Step}
we obtain
$$ \mathbb{E}_{(x,W)} \left[\left( \sup_{v \in W^{\perp}, \| v \| = 1} \| P_{W_{m+1}^{\perp}}(  D_x f^{(m+1)N}_\omega( v) ) \| /  
  \con(D_x f_\omega^{(m+1)N}|W) \right)^{\sigma}\right]\leq   $$
  $$ \mathbb{E}_{(x,W)} \left[\left( \frac{\DS \sup_{v \in W^{\perp}, \| v \| = 1} \| P_{W_{m}^{\perp}}(  D_x f^{mN}_\omega( v) ) \|}   
  {\con(D_x f_\omega^{mN}|W)} \right)^{\sigma}
  \left( \frac{\DS \sup_{v \in W_m^{\perp}, \| v \| = 1} \| P_{W_{m+1}^{\perp}}(  D_{X_m} f^{Nm, N}_\omega( v) ) \|}   
  {\con(D_{X_m} f_\omega^{Nm, N}|W_m)} \right)^{\sigma}
  \right] =  $$
 $$ \mathbb{E}_{(x,W)} \left[\left( \frac{\DS \sup_{v \in W^{\perp}, \| v \| = 1} \| P_{W_{m}^{\perp}}(  D_x f^{mN}_\omega( v) ) \|}   
  {\con(D_x f_\omega^{mN}|W)} \right)^{\!\!\!\sigma} \!\!\!
  \mathbb{E}_{(x, W)}\!\!\left[\left( \frac{\DS \sup_{v \in W_m^{\perp}, \| v \| = 1} \| P_{W_{m+1}^{\perp}}(  D_{X_m} f^{Nm, N}_\omega( v) ) \|}   
  {\con(D_{X_m} f_\omega^{Nm, N}|W_m)} \right)^{\!\!\sigma}\!\!\!\Big|\cF_{mN}\right]
  \right]   $$ 
$$ =\mathbb{E}_{(x, W)} \left[\left( \frac{\DS \sup_{v \in W^{\perp}, \| v \| = 1} \| P_{W_{m}^{\perp}}(  D_x f^{mN}_\omega( v) ) \|}   
  {\con(D_x f_\omega^{mN}|W)} \right)^{\sigma}  \gamma_N(X_m, W_m)\right]\leq
  $$
$$ e^{-N\sigma\kappa_1/2} \mathbb{E}_{(x, W)} \left[\left( \frac{\DS \sup_{v \in W^{\perp}, \| v \| = 1} \| P_{W_{m}^{\perp}}(  D_x f^{mN}_\omega( v) ) \|}   
  {\con(D_x f_\omega^{mN}|W)} \right)^{\sigma}  \right]\leq e^{-N(m+1)\sigma\kappa_1/2},
  $$  
where the last step uses the inductive assumption.  This completes the inductive step.\\

{\bf Step 3.} Let $\DS D_1 = \sup_{g \in \supp(\mu)} \max( \| Dg \|, \| Dg^{-1} \| ).$ Given $n$, let $m$ be a number such that\\
$mN\leq n<(m+1) N$. Then
	$$
	\int  \Big ( \sup_{v \in W^{\perp}, \| v \| = 1} \| P_{D_xg( W)^{\perp}}(  D_xg( v) ) \| /  {\con(D_xg|W)} \Big )^{\sigma} d\mu^{* n}(g) \leq
    $$$$
  D_1^{2N\sigma} \int  \Big ( \sup_{v \in W^{\perp}, \| v \| = 1} \| P_{D_xg( W)^{\perp}}(  D_xg( v) ) \| /  { \con(D_xg|W)} \Big )^{\sigma} d\mu^{* mN}(g) \leq $$$$ 
  D_1^{2N\sigma} e^{-m\kappa_1\sigma N/2}\leq D_1^{4N\sigma} e^{-n\kappa_1\sigma/2}
$$
where the second inequality relies on \eqref{ExpGap}. Setting $C_1=D_1^{4N\sigma}$ 
we obtain the result.
\end{proof}

    Recall that the topology on the Grassmannians is generated by the following systems of fundamental neighborhoods: 
$$\cU_\delta(W)=\{W'\mid W'=\graph(\cL) \text{ where } \cL\colon W\to W^\perp \text{ satisfies }  \|\cL\|\leq \delta\}.  $$

   The following result is proven in \S \ref{SSA2}, and gives a quantitative estimate on how likely it is for a most contracted direction to live in a ball of radius $\rho$.
    \begin{lemma}
    \label{sublem growth}
     There are constants 
        $   \kappa_2, C_2, c_3, c_4, \rho_0 >0$ such that
        for any $0 < \rho \le \rho_0$,  there exists $n_1\le C_2\abs{\log \rho}$, such that for any $x \in M$, any $W \in \Gr(T_x M, d- 1)$, there exists a set $\Omega_0=\Omega_0(W, \rho)\subset \Omega$ such that
        $\DS \PP(\Omega_0^c)\leq \rho^{\kappa_2}$
        and for each $\DS W'\in \cU_{c_3 \rho}(W)$
        for every unit vector $u  \in  W'$ and every $\omega\in\Omega_0$ we have for any $n > |\log \rho|$ divisible by $n_1$, 
     \begin{equation}
         \label{eq Dxfmngrowth}
        \| D_xf_{\omega}^n(u) \|  >    \con(D_x f_\omega^n) e^{c_4 n}.
        \end{equation} 
	\end{lemma}

\begin{proof}[Proof of Lemma \ref{lem:ahlfors_Ess}]
Let $x$ be a volume typical point, and let $0<\varepsilon< c_4/2$.
By the Multiplicative Ergodic Theorem, for a.e.\ $\omega$, we have that for large $n$
$$
\con(D_xf_\omega^n)\geq e^{(\lambda^{(1)}(x) -\eps)n} 
\text{ and }
\|D_x f_\omega^n|E^{ss}\|\leq e^{(\lambda^{(1)}(x) +\eps)n}, 
$$
where $\lambda^{(1)}(x)$ is the smallest Lyapunov exponent. Accordingly for large $n$ 
$$ \|D_x f_\omega^n(v)\|\leq \con(D_x f_\omega^n) e^{2\eps n} \| v \|, \quad \forall v\in E^{ss}(x, \omega) \setminus \{0\}.  $$
Now \eqref{eq Dxfmngrowth} shows that for $\omega\in \Omega_0(W, \rho)$, $W'\in \mc{U}_{c_3\rho}(W)$, $E^{ss}(\omega, x)\cap W'
=\{0\}$. The conclusion follows.
\end{proof}

\begin{remark}
 Note that Lemma \ref{lem:ahlfors_Ess} shows in particular that $V^{(1)}$ is one dimensional since otherwise it would intersect 
$W$.
\end{remark}

\subsection{Controlling small balls in Grassmannians.} 
\label{SSA2}
	\begin{proof}[Proof of Lemma \ref{sublem growth}]
    Recall the notation from \eqref{AppMC} and \eqref{OrbitNotation}.
	Denote
    $\DS
    	D_1 = \!\!\!\sup_{g \in \supp(\mu)} \!\!\!\max( \| Dg \|, \| Dg^{-1} \| ),
    $
    and   $K = 2 ( D_1 + 1 )$.
    	Take $n_1 $ so that 
	\begin{equation}
    \label{DefN1}
	   n_1 \leq	| \log \rho | / (3 K )  \leq 2 n_1. 
	\end{equation}

  We define $(X_m)_{m \geq 0}, (W_m)_{m \geq 0}$ as in \eqref{OrbitNotation} for $N = n_1$.
    Under splittings $T_{X_m} M = W_m \oplus W_m^{\perp}$ and $ T_{X_{m+1}} M = W_{m+1} \oplus W_{m+1}^{\perp}$, we may write 
	\aryst
	D_{X_m} f^{mn_1, n_1}_{\omega} = \begin{bmatrix}
		A_m & C_m \\ 0 & B_m
	\end{bmatrix}.
	\earyst
	By   Lemma \ref{lem MomentBound} we have  
	\begin{equation}
	\sup_{x \in M, W \in \Gr(T_x M, d-1)} \EV_{(x, W)}( 	(\| A_0^{-1} \| \| B_0 \|)^{\sigma}  ) \leq C_1 e^{ - n_1 \sigma \kappa_1 / 2 } < e^{ - n_1 \sigma \kappa_1 / 3 },
	\end{equation}
	provided that $n_1$ is large enough.
    
  Similarly to \eqref{ExpGap} we see that
	\ary \label{eq supxEEVproduct}
	\sup_{x \in M, W \in \Gr(T_x M, d-1)} \EV_{(x, W)}
    \left(  \prod_{j = 0}^{m-1} (\| A_j^{-1} \| \| B_j \|)^{\sigma}  \right) \leq  e^{ - m n_1 \sigma \kappa_1 / 3 }.
	\eary
	
	For each $m \geq 1$, we define the event 
	\aryst
	\Omega_m = \{ \omega \in \Omega   \mid     \prod_{j = 0}^{m-1} \| A_j^{-1} \| \| B_j \|    >  2^{-m}e^{- m n_1 \kappa_1 / 4 } \}.
	\earyst
	By Markov's inequality,  \eqref{eq supxEEVproduct} and by letting $n_1$ be sufficiently large, we deduce that 
	\[
	\PP(\Omega_m) \leq   2^{m \sigma} e^{- m n_1 \sigma \kappa_1 / 3 + m n_1 \sigma \kappa_1 / 4} \leq   e^{ - m n_1 \sigma \kappa_1 / 20}.
	\]
	Define 
	$\displaystyle
	\Omega_0 \coloneqq ( \bigcup_{j = 1}^{\infty} \Omega_j  )^{c}.
	$
	If $n_1$ is sufficiently large we get
	\begin{equation}
	\PP(\Omega_0^c) \leq \sum_{j = 1}^{\infty} \PP(\Omega_j) \leq  e^{ - n_1 \sigma \kappa_1 / 30}.
	\end{equation}

    Having now set up the set $\Omega_0$, we will check that \eqref{eq Dxfmngrowth} holds for $\omega\in \Omega_0$.
	  Let
    $\cL\colon W \to W^{\perp}$ be an arbitrary linear map with $\| \cL \| \leq c_3 \rho$ and $u \in W$ be a unit vector.
	Denote $\cL_0 = \cL$ and $U_0 = u$. For each integer $m \geq 0$, define
	\aryst
	\cL_{m+1} = B_m  \cL_m   (A_m + C_m \cL_m)^{-1}, \quad U_{m+1} = (A_m + C_m \cL_m) U_m.
	\earyst
	We see that for integer $m \geq 0$,  
	\aryst
	D_{X_m} f^{mn_1, n_1}_{\omega}( U_m + \cL_m(U_m) ) = U_{m+1} + \cL_{m+1}(U_{m+1}).
	\earyst 
	
	We have   
	\aryst
	\| A_j^{-1} \|, \| C_j \| \leq e^{D_1 n_1}, \quad \forall j \geq 0.
	\earyst

	  Note that for $c_3\le 1$ it follows  from \eqref{DefN1} that $\rho \leq e^{ - 3 K n_1} c_3^{-1}$. We will show by induction that for every $\omega \in \Omega_0$, for every integer $m \geq 0$, we have
	\ary \label{eq Lmbound}
	\|  \cL_m \| \leq c_3 \rho 2^m  \prod_{j = 0}^{m-1} \| A_j^{-1} \|\|B_j \|   \leq 
    e^{- m n_1 \kappa_1 / 4}  c_3 \rho.
	\eary
	The equality is true for $m = 0$.  Assume that it holds for some $m \geq 0$, then 
	\aryst
	\|  A_m^{-1}  C_m \cL_m \| < e^{2 D_1 n_1}  \cdot c_3 \rho \leq e^{2 D_1 n_1} \cdot e^{- 3 K  n_1}  <  1/2.
	\earyst
	Notice that  
	\aryst
	\| \cL_{m+1} \| &\leq&  \| B_m\mc{L}_m\|\|A_m^{-1}\| \| ( {\rm Id}  + A_m^{-1} C_m \cL_m )^{-1} \|
    \\&\leq &
    2  
    \| 
    A_m^{-1}\|\| B_m 
    \|   \| \cL_m \| \leq c_3 \rho 2^{m+1}  \prod_{j = 0}^{m} \| A_j^{-1} 
     \|\|B_j \|    \leq e^{ - (m+1) n_1 \kappa_1 /4} c_3 \rho. 
	\earyst 
	The last inequality above follows since $\omega \in \Omega_0 \subset \Omega_{m+1}^{c}$.
	This finishes the proof of \eqref{eq Lmbound} by induction.
 	By a similar induction, we see that 
	\ary \label{eq Umbound}
	\| U_{m} \| \geq 2^{-m} \prod_{i = 0}^{m - 1}\| A_{i}^{-1} \|^{-1}.
	\eary

	Now suppose $n >  | \log \rho |$.	
	Since $K > 1$, we have $n  > 3n_1$.   By letting $\rho \ll_{\mu} 1$, we have $n_1 \gg_{\mu} 1$ 
    and $1-\|\mc{L}_m\|>1/2$ for all $m$ by \eqref{eq Lmbound}.
 For every $\omega \in \Omega_0$,  it is now straightforward to deduce from \eqref{eq Umbound} and the definition of $\Omega_0$,  that for every unit $u \in W$,
  \[
  \|D_x f^{mn_1}_{\omega}( u + \cL(u) )\|\ge \|U_m\|-\|\mc{L}_m\|\|U_m\|\ge  e^{mn_1 \kappa_1/4} \prod_{i = 0}^{m-1} \|B_i\| \geq e^{mn_1 \kappa_1/4} \con(D_xf^{mn_1}_{\omega}).
 \]
 Thus \eqref{eq Dxfmngrowth} holds with $c_4=\kappa_1/4$.
	To finish the proof, it suffices to take
	\[
	{\kappa_2} = \sigma {  \kappa_1} /(200K),
	\]
	since by the lower bound $\abs{\log\ \rho}/6K\le n_1$ in \eqref{DefN1} we have
	$\DS
	\PP(\Omega_0^c) \leq   e^{ -  n_1 \sigma \kappa_1 / 30} \leq \rho^{\kappa_2}.
	$	
	\end{proof}

\subsection{Proof of Proposition \ref{prop:tail_on_splitting}}

Let $x \in M$ be a $\vol$-typical point, and let  $W \in \Gr(T_x M, d-1)$.
Let $\omega \in \Omega$ be such that $E^{ss}(\omega, x)$ is defined. 
Define for each $n \geq 0$, 
$$x_n = f^n_{\omega}(x),\quad V_n = D_xf^n_{\omega}(E^{ss}(\omega, x)), \quad W_n = D_xf^n_{\omega}(W), \quad  \theta_n=\angle(V_n, W_n). $$
\begin{lemma} \label{lem simplelinearalgebra}
	Assume that for all integers $j, k \geq 0$ 
     and some positive reals $\fk,\fe $ with $\fe<\fk/8$
    we have
	\ary 
	&&  \theta_k > e^{ - C} e^{- k \fe}, \label{eq condition1} \\
	&& \sup_{v  \in W_k^{\perp}, \| v \| = 1} \| P_{W_{k+j}^{\perp}}(  D_{x_k} f_{\omega}^{k, j}( v) ) \| <  e^{C }  e^{ - j \fk / 4 + k \fe}   \con(D_{x_k} f_{\omega}^{k, j} \vert W_k). \label{eq condition2}
	\eary 
	Then \eqref{eqn:subtempered_splitting_norm} and \eqref{eqn:subtempered_splitting_angle} hold with 
    $ \fC=2C, \eps=2\fe, \kappa=\frac{\fk}{8}$; and 	\eqref{eq itm3} holds 
    with  
    $ \fC'=2C, \eps'=2\fe, \kappa'=\frac{\fk}{16}$.
\end{lemma}

\begin{proof}
	Since \eqref{eqn:subtempered_splitting_angle}  and \eqref{eq condition1} are the same, we just need to verify \eqref{eqn:subtempered_splitting_norm} and \eqref{eq itm3}. 
    
    To check \eqref{eqn:subtempered_splitting_norm},   let $u_k$ be a unit vector orthogonal to $W_k$ and $v_k$ be a unit vector in $V_k$.
Splitting $v_k=w_k \cos \theta_k+u_k \sin \theta_k$ with $w_k\in W_k$ we obtain \\
\[D_{x_k} f_{\omega}^{k,j} v_k=\sin\theta_k D_{x_k} f_{\omega}^{k,j} u_k+w_{k,j}\] where $w_{k,j}\in W_{k+j}.$
Hence projecting the previous line onto $W_{j+k}^{\perp}$
\begin{equation}
\label{eq simpleineq}
\sin\theta_{k+j}\DS \left\|D_{x_k} f_{\omega}^{k, j}(v_k)\right\|=\sin\theta_k \left\|P_{W_{k+j}^{\perp}} D_{x_k} f_{\omega}^{k,j} u_k\right\|.
\end{equation}
Therefore by \eqref{eq condition2}, we obtain
\[
\left\|D_{x_k} f_{\omega}^{k, j}(v_k)\right\|\le \frac{\sin \theta_k}{\sin \theta_{k+j}} 
e^Ce^{-j\fk/4+k\fe}\con(D_{x_k} f^{k,j}_{\omega}\vert W_k).
\]
Thus if $\fe<\fk/8$, then \eqref{eqn:subtempered_splitting_norm} follows from \eqref{eq condition1} and \eqref{eq condition2} with the given constants.

    We now prove  \eqref{eq itm3}. Since $\det(D_x f_{\omega}^{k, j})\!=\! 1$, we get
	\begin{equation}
 \DS \| P_{W_{k+j}^{\perp}}(  D_{x_k} f_{\omega}^{k, j}{ u_k} ) \|   \Big (  \con\left(D_{x_k} f_{\omega}^{k, j}  \vert W_k \right) \Big )^{d-1} \!\!\leq \!1.
	\end{equation}
	By \eqref{eq condition2},
$\DS 
\| P_{W_{k+j}^{\perp}}(  D_{x_k} f_{\omega}^{k, j}{ u_k} ) \|    \leq (e^{C }  e^{ - j \fk / 4 + k \fe})^{\frac{d-1}d}.$
	Now \eqref{eq itm3} follows from \eqref{eq simpleineq}.
\end{proof}

\begin{proof}[  Proof of Proposition \ref{prop:tail_on_splitting}]
	By Lemma \ref{lem simplelinearalgebra}, we have
	\aryst
	\PP(\mc{A}^c) \leq \sum_{  k \geq 0} \PP(\mbox{\eqref{eq condition1} fails for $k$}) + \sum_{j, k \geq 0} \PP(\mbox{\eqref{eq condition2} fails for $(j, k)$}), 
	\earyst
   where in \eqref{eq condition1} and \eqref{eq condition2} we have  $\fk=16\kappa$,
 $\fe=\epsilon/2$, $C=\fC/2$.

    Let $\alpha, C' $ be given by Lemma \ref{lem:ahlfors_Ess}. 
	By Lemma \ref{lem:ahlfors_Ess}, we see that for some $c' > 0$
	\aryst
	\sum_{  k \geq 0} \PP(\mbox{\eqref{eq condition1} fails for $k$})  \leq C'	e^{ - \fC \alpha/2} \sum_{k \geq 0} e^{ - k \epsilon \alpha/2} \leq 1/(2c') e^{- c'  \fC }.
	\earyst
By Lemma \ref{lem MomentBound}, we see that for some $c'' > 0$
	\[
	\sum_{j, k \geq 0} \PP(\mbox{\eqref{eq condition2} fails for $(j, k)$})   \leq  	 e^{- C \sigma} \sum_{j, k \geq 0}  e^{ -  \sigma k \fe - j \sigma (\kappa_1/2 -\fk/4)} \leq 1 /(2c') e^{ -  c'' \fC },
	\]
  provided that $\kappa$ (and, hence $\fk$) are small enough.
    The   result follows from the two inequalities above.
\end{proof}

\printbibliography

\end{document}